\documentclass[11pt]{amsart}
\usepackage{longtable}
\usepackage{array}
\usepackage{multicol}
\usepackage{graphicx}
\usepackage[colorlinks=true,citecolor=black,linkcolor=black,urlcolor=blue]{hyperref}

\makeatletter
 \def\@textbottom{\vskip \z@ \@plus 1pt}
 \let\@texttop\relax
\makeatother
\usepackage{amsmath, amssymb, amsbsy, amsfonts, amsthm, latexsym, amsopn, amstext, amsxtra, euscript, amscd, color, mathrsfs}
\usepackage[normalem]{ulem}
\usepackage{soul}

\usepackage{cite}

\makeatletter

\@namedef{subjclassname@2020}{%
  \textup{2020} Mathematics Subject Classification}
\makeatother

\PassOptionsToPackage{hyphens}{url}\usepackage{hyperref}
 
 \usepackage[capbesideposition=outside,capbesidesep=quad]{floatrow}

\restylefloat{table}
\restylefloat{table}
         
\usepackage{multirow,caption}
            
\usepackage{amscd}
\usepackage{color,enumerate}

\newcommand{\RNum}[1]{\lowercase\expandafter{\romannumeral #1\relax}}

\usepackage[colorinlistoftodos,prependcaption,textsize=tiny]{todonotes}

\theoremstyle{plain}
\newtheorem{thm}{Theorem}[section]
\newtheorem{lem}[thm]{Lemma}
\newtheorem{cor}[thm]{Corollary}
\newtheorem{prop}[thm]{Proposition}

\newtheorem{rmk}[thm]{Remark}

\newtheorem{thm-con}[thm]{Theorem-Conjecture}
\numberwithin{equation}{section}

\theoremstyle{definition}
\newtheorem{defn}[thm]{Definition}

\newcommand{\F}{\mathbb F}

\begin{document}
\title[Local permutation polynomials, their companions, and related enumeration results]{Bivariate local permutation polynomials, their companions, and related enumeration results}

\author[S.U. Hasan]{Sartaj Ul Hasan}
\address{Department of Mathematics, Indian Institute of Technology Jammu, Jammu 181221, India}
\email{sartaj.hasan@iitjammu.ac.in}

   \author[R. Kaur]{Ramandeep Kaur}
  \address{Department of Mathematics, Indian Institute of Technology Jammu, Jammu 181221, India}
  \email{2022rma0027@iitjammu.ac.in}

\author[H. Kumar]{Hridesh Kumar}
\address{Department of Mathematics, Indian Institute of Technology Jammu, Jammu 181221, India}
\email{2021rma2022@iitjammu.ac.in}

 \thanks{The first named author acknowledges partial support from the Core Research Grant CRG/2022/005418, provided by the Science and Engineering Research Board, Government of India. The second and third named authors are supported by the Prime Minister’s Research Fellowship, under PMRF IDs 3003658 and 3002900, respectively, at IIT Jammu.}

\begin{abstract}
We construct a new family of permutation group polynomials over finite fields of arbitrary characteristic, which are special types of bivariate local permutation polynomials. For this family, we explicitly construct their companion. We also determine the total number of permutation group polynomials of this form. Moreover, we resolve the problem of enumerating $e$-Klenian polynomials over finite fields for $e\geq 1$, a problem previously noted as nontrivial by Gutierrez and Urroz (2023). In addition, we provide the exact number of permutation group polynomials equivalent to our proposed permutation group polynomials, as well as the exact number of those permutation group polynomials equivalent to $e$-Klenian polynomials.
\end{abstract}
\keywords{Finite fields, permutation polynomials, local permutation polynomials, Latin squares, symmetric group}
\subjclass[2020]{12E20, 11T06, 11T55}
\maketitle
\section{Introduction}\label{S1}
Let $\F_{q}$ denote the finite field with $q$ elements, where $q$ is a power of a prime, and let $\F_q[X_1, \ldots,X_m]$ represent the ring of polynomials in $m$ variables $X_1, \ldots, X_m$ over $\F_q$, with $m$ being a positive integer. It is well-known~\cite{LNH_1997} that any map from $\F_q^m$ to $\F_q$ can be uniquely expressed as a polynomial $f(X_1,X_2,\ldots,X_m) \in \F_q[X_1,X_2,\ldots,X_m]$ of degree strictly less than $q$ in each variable. Consequently, the study of functions from $\F_q^m$ to $\F_q$ is same as the  study of polynomials in $m$ variables over $\F_q$ with degree strictly less than $q$ in each variable. Thus, in what follows, we will consider only polynomials in several variables over $\F_q$. A polynomial $f\in\F_q[X_1,X_2,\ldots,X_m]$ is called a permutation polynomial (PP) if the equation $f(X_1,X_2,\ldots,X_m)=a$ has exactly $q^{m-1}$ solutions in $\F_q^{m}$ for each $a\in \F_q$. For $m=1$, the above definition coincides with the classical notion of a permutation polynomial in one variable. While permutation polynomials in one variable have a long history, the systematic study of permutation polynomials in several variables can be traced back to the seminal work of Niederreiter \cite{NH_1970}, where he classified all the PPs in several variables of degree at most two.

This paper investigates local permutation polynomials over finite fields. This concept was first introduced by Mullen~\cite{Mullen_1980, Mullen_1_1980}, who established necessary and sufficient conditions for polynomials in two and three variables to be local permutation polynomials over prime finite fields. A polynomial $f(X_1,X_2,\ldots,X_m) \in\F_q[X_1,X_2,\ldots,X_m]$ is called a local permutation polynomial (LPP) if for each $ i \in \{1, \ldots, m\}$ and for each tuple $\overline{a}_i:=(a_1,\ldots,a_{i-1},a_{i+1},\ldots,a_m)\in \F_q^{m-1}$, the univariate polynomial $f(a_1,\ldots,a_{i-1},X_i,a_{i+1},\ldots,a_m)\in \F_q[X_i]$ is a permutation polynomial over $\F_q$. Note that an LPP is always a permutation polynomial; however, the converse may not be true in general. For instance, $f(X_1,X_2)=X_2 \in \F_q[X_1,X_2]$ is a permutation polynomial, but it is not an LPP over $\F_q$. 

The notion of orthogonality of polynomials in several variables over a finite field was introduced by Niederreiter~\cite{NH_1971}. However, we restrict our discussion to the case of two variables, as our primary focus is on studying LPPs in two variables.

Two LPPs $f_1$ and  $f_2 \in \F_q[X_1,X_2]$ are said to be orthogonal (or companion) of each other if and only if the system of equations 
\begin{equation*}
\begin{cases}
f_1(X_1,X_2)=a,\\
 f_2(X_1,X_2)=b
\end{cases}
\end{equation*}
has a unique solution $(x_1,x_2) \in \F_q \times \F_q$ for all $ (a,b) \in \F_q \times \F_q$. 

Gutierrez and Urroz, in their seminal paper~\cite{JJ_2023}, established a one-to-one correspondence between bivariate LPPs and a specific type of $q$-tuples of permutation polynomials in one variable over the finite field with $q$ elements. This correspondence provides a beautiful algebraic framework for investigating bivariate LPPs over finite fields.  In particular, if the elements of the $q$-tuple form a subgroup of $\mathfrak{S}_q$, the symmetric group on $q$ elements of $\F_q$, then the corresponding LPP is referred to as a permutation group polynomial. Using this correspondence, the authors in~\cite{JJ_2023} constructed a family of bivariate LPPs based on a class of symmetric subgroups without fixed points, referred to as $e$-Klenian polynomials. Motivated by the work of Gutierrez and Urroz~\cite{JJ_2023}, Hasan and Kumar~\cite{HK} constructed three new families of permutation group polynomials and provided their corresponding companions. Additionally, they investigated companions for $e$-Klenian polynomials over finite fields of even characteristic.
 
The importance of bivariate LPPs lies in the fact that there is a one-to-one correspondence~\cite{Mullen_book} between the set of Latin squares of order $q$, where $q$ is a prime power, and the set of bivariate LPPs over the finite field of order $q$. A Latin square of order $n$ is an $n \times n$ matrix with entries from a set $S$ of cardinality $n$, such that every element of $S$ appears exactly once in each row and each column.  Two Latin squares are said to be orthogonal if, when superimposed, they produce all the ordered pairs of $S \times S$. A set of Latin squares of order $n$ that are pairwise orthogonal is referred to as Mutually Orthogonal Latin Squares (MOLS). An orthogonal Latin square of the Latin square associated with a particular bivariate LPP corresponds to a companion of that LPP. Consequently, the study of LPPs (or their companions) over $\F_q$ is essentially equivalent to the study of Latin squares (or orthogonal Latin squares) of order $q$. Constructing Mutually Orthogonal Latin Squares (MOLS) is known to be a challenging combinatorial problem. However, the companion of a given bivariate LPP provides an elegant algebraic method for constructing MOLS~\cite{JJ_2023}. Latin squares and MOLS have a wide range of applications in several areas such as combinatorial design theory \cite{MM_2017}, cryptography \cite{SS_1992} and coding theory \cite{KD_2015, MGFL_2020,WW_2014}.
  
In the existing literature, only two papers~\cite{JJ_2023, HK} focus on permutation group polynomials over finite fields. In this work, we construct a new family of permutation group polynomials based on a specific subgroup of the symmetric group $\mathfrak{S}_q$. As noted earlier, two bivariate LPPs are companions if and only if their associated Latin squares are orthogonal. While finding an explicit expression for a companion of a given bivariate LPP is generally a challenging problem, we explicitly provide companion for the family of LPPs introduced in this paper.

Two permutation group polynomials are of the same form if their associated groups, as subgroups of $\mathfrak{S}_q$, are isomorphic. One of the longstanding and challenging problems in combinatorics is the enumeration of all Latin squares of large order $n$. Euler~\cite{EulerT} first addressed this problem, and despite several attempts -- see, for example,~\cite{HKO, KD_2015, MW} and references therein -- the exact count of Latin squares has been determined only for orders up to $n=11$. As permutation group polynomials are a special class of bivariate LPPs, which are inherently linked to Latin squares, enumerating these polynomials over the finite field $\F_q$ is equivalent to enumerate a specific subclass of Latin squares of order $q$. This paper provides the enumeration of all Latin squares that can be constructed from the proposed family of permutation group polynomials. More precisely, we determine the exact number of permutation group polynomials that arise in the form constructed in this paper, and separately determine the number of permutation group polynomials that are equivalent to this family. Gutierrez and Urroz [5] determined the exact number of $0$-Klenian polynomials over $\F_q$ and observed that enumerating those $e$-Klenian polynomials for $e \geq 1$ is nontrivial. In this paper, we present a complete solution to this problem. We also give the exact number of permutation group polynomials that are equivalent to $e$-Klenian polynomials.

The remainder of this paper is organized as follows. Section~\ref{S2} introduces key definitions and lemmas. In Section~\ref{S3}, we construct a new family of permutation group polynomials and demonstrate that they are not equivalent to the known classes of permutation group polynomials. Furthermore, we determine the companions for these  permutation group polynomials. Section~\ref{S4} is devoted to counting results for the permutation group polynomials constructed in Section~\ref{S3}, as well as for $e$-Klenian polynomials over $\F_q$ with $e \geq 1$.
Finally, we conclude our paper in Section \ref{S5}. 

\section{Preliminaries}\label{S2}
In this section, we introduce preliminary definitions and results needed for the subsequent sections. Throughout this paper, the elements of $\F_q$ are listed as $\F_q = \{c_0, c_1, \ldots, c_{q-1}\}$. A $q$-tuple $(\beta_0, \ldots, \beta_{q-1}) \in {\mathfrak S}_q^q$ is referred to as a permutation polynomial tuple if, for $0\leq i \neq j \leq q-1$, the composition $\beta_i \beta_j^{-1}$ has no fixed points.

The following lemma provides a one-to-one correspondence between the set of bivariate LPPs over $\F_q$ and the set of permutation polynomial tuples in ${\mathfrak S}_q^q$.
\begin{lem}\cite[Lemma 7]{JJ_2023}
\label{tpl}
There is a bijection between the set of bivariate local permutation polynomials $f\in \F_q[X_1,X_2]$ and the set of $q$-tuples $\underline{\beta}_f :=(\beta_0,\ldots,\beta_{q-1}) \in {\mathfrak S}_q^q$ of permutations of $\mathbb F_q$ such that $\beta_i \beta_j^{-1}$ has no fixed points in $\F_q$ for $0 \leq i \neq j \leq q-1$. Furthermore, $f$ and $\underline{\beta}_f$ are associated to each other by the following relation: for each $0 \leq i \leq q-1$, $f(x,\beta_i(x))=c_i$ for all $x\in \F_q$.
\end{lem}

Let $(\beta_0, \beta_1, \ldots, \beta_{q-1}) \in {\mathfrak S}_q^q$ be a $q$-tuple whose elements form a subgroup of ${\mathfrak S}_q$. Then, it follows that $(\beta_0, \beta_1, \ldots, \beta_{q-1})$ constitutes a permutation polynomial tuple if and only if no $\beta_i$ (for $i=0,1,\ldots, q-1$), other than the identity permutation, fixes any element of $\F_q$. In this case, the corresponding local permutation polynomial is called a permutation group polynomial. 

As a result, by Lemma~\ref{tpl}, constructing a permutation group polynomial is equivalent to identifying a subgroup of ${\mathfrak S}_q$ of order $q$ such that no non-identity permutation of this subgroup fixes any element of $\F_q$. Using this, Gutierrez and Urroz~\cite[Corollary 17]{JJ_2023} introduced a family of permutation group polynomials, referred to as $e$-Klenian polynomials. The enumeration of $e$-Klenian polynomials will be discussed in Section~\ref{S4}. For completeness, their result is restated in the following lemma.

\begin{lem}\cite[Corollary 17]{JJ_2023}
\label{JJ_tpl}
Let $q=p^r$, where $p$ is a prime number and $r$ is a positive integer. Let $0 \leq e < r$, $\ell =p^e$ and  $t=\frac{q}{\ell}$. Let $a=C_{0,a} \cdots C_{t-1,a}$, where
\[
C_{i,a}=(c_{i\ell},c_{i\ell+1}, \ldots, c_{(i+1)\ell-1})~\mbox{for all}~0\leq i \leq t-1
\]
and $b=C_{0,b} \cdots C_{\ell-1,b}$, where
\[
C_{j,b}=(c_{j},c_{j+\ell}, \ldots, c_{j+(t-1)\ell})~\mbox{for all}~0\leq j \leq \ell-1.
\]
Then the set defined by \( G = \{a^{\tilde{i}} b^{\tilde{j}} : 0 \leq \tilde{i} \leq \ell - 1, 0 \leq \tilde{j} \leq t - 1 \} \) is a subgroup of \( \mathfrak{S}_q \), such that each permutation of \( G \), except the identity, has no fixed points, and \( |G| = q \), where \( |G| \) denotes the order of the group \( G \).
\end{lem}
The following definitions and lemma will play a crucial role in determining the companion of the family of permutation group polynomials constructed in the subsequent section.
\begin{defn}\label{D21}
 Let $h_1,h_2\in\F_q[X]$ be two permutation polynomials. Let $A :=\{(c,h_1(c)):c \in \F_q\}$ and $B:=\{ (c,h_2(c)):c \in \F_q\}$. We say that $h_1$ intersects $h_2$ simply  if $|A \cap B|=1$. 
 \end{defn}
 
 \begin{defn}\label{D22}
Let $f\in\F_q[X_1,X_2]$ be an LPP and $\underline{\beta}_f=(\beta_0,\ldots,\beta_{q-1}) \in {\mathfrak S}_q^q$ be the corresponding permutation polynomial tuple. We say that a univariate permutation polynomial $h$ intersects the bivariate LPP $f$ simply if $h$ intersects $\beta_i$ simply for each $0 \leq i \leq q-1$. 
 \end{defn}
 
 \begin{lem}\label{L21}\cite[Lemma 4.4]{HK}
Let $f\in \F_q[X_1,X_2]$ be a permutation group polynomial and $h\in\F_q[X]$ be a permutation polynomial which intersects $f$ simply. Let $\underline{\beta}_f=(\beta_0,\ldots,\beta_{q-1})$ be the permutation polynomial tuple corresponding  to $f$. Then the polynomial $g$ associated  to $\underline{\gamma}_g=(h \beta_0, \ldots,h \beta_{q-1})$ is a companion of $f$.
\end{lem}

\section{Permutation group polynomials and their companions}\label{S3}
As discussed in the previous section, permutation group polynomials can be constructed by identifying subgroups of ${\mathfrak S}_q$ of order $q$, where no non-identity permutation of the subgroup fixes any element of $\F_q$. In this section, we use this technique to construct a new family of permutation group polynomials over finite fields of arbitrary characteristic. We also demonstrate that these permutation group polynomials are not equivalent to previously known classes of permutation group polynomials. In addition, we also provide the explicit expressions for their companions. 

\begin{thm}\label{T31}
Let \( q = p^n \), where \( p \) is a prime number and \( n \ge 2 \) is a positive integer. Let \( \mathbb{F}_q = \{ c_0, c_1, \ldots, c_{q-1} \} \) denote the finite field with \( q \) elements. For $\delta \in \{1,2\}$, define
\[
a_0:=\displaystyle \prod_{j_0=0}^{p^{n-\delta}-1}(c_{j_0p^{\delta}},c_{j_0p^{\delta}+1},\ldots,c_{j_0p^{\delta}+p^{\delta}-1})
\]
and for $i \in \{1,2,\ldots,n-\delta\}$
\[
a_i:=\displaystyle \prod_{j_i<p^{i+\delta-1} \pmod {p^{i+\delta}}}(c_{j_i},c_{j_i+p^{i+\delta-1}},\ldots,c_{j_i+(p-1)p^{i+\delta-1}}).
\]
 Then the set
\[
G = \{ a_0^{i_0} a_1^{i_1} \cdots a_{n-\delta}^{i_{n-\delta}} \mid 0 \leq i_0 \leq p^{\delta}-1, 0 \leq  i_1, \ldots, i_{n-\delta} \leq p-1 \}
\]
is an Abelian subgroup of the symmetric group $\mathfrak{S}_q$ of order $q$. Moreover, no permutation of $G$, except the identity, has any fixed point in $\mathbb{F}_q$.  
\end{thm}
\begin{proof}
To prove that $G$ is a subgroup of $\mathfrak{S}_q$, it suffices to show $a_ha_k = a_ka_h$ for all $h,k \in \{0,1,\ldots,n-\delta\}$. Before establishing this commutativity, we explicitly write the definitions of the permutations $a_i$'s and their powers. The cyclic structure of $a_0$ suggests that for any $0 \leq j_{0} \leq p^{n-\delta}-1$ and $0 \leq r_{0} \leq p^{\delta}-1$, we have
\begin{equation}\label{e31}
a_{0}(c_{j_{0}p^{\delta}+r_{0}})=\begin{cases}
c_{j_{0}p^{\delta}+r_{0}+1}& \text{ if } r_{0} < p^{\delta}-1, \\
c_{j_{0}p^{\delta}+r_{0}+1-p^{\delta}}    & \text{ if } r_{0} =p^{\delta}-1
       \end{cases}=c_{(r_{0}+1) \pmod {p^{\delta}}+{j_0}p^{\delta}}.
\end{equation}

Furthermore, for each  $ j_{i}<p^{i+\delta-1} \pmod {p^{i+\delta}}$ and $0 \leq r_{i} \leq p-1$, the cyclic structure of $a_i$, where $1 \leq i \leq n-\delta$, yields
\begin{equation}\label{e32}
a_i(c_{j_{i}+r_{i}p^{i+\delta-1}})=\begin{cases}
c_{j_{i}+(r_{i}+1)p^{i+\delta-1}}& \text{ if } r_{i} < p-1, \\
c_{j_{i}}& \text{ if } r_{i} =p-1
       \end{cases}=c_{((r_{i}+1)\pmod p)p^{i+\delta-1}+j_i}.
\end{equation}
Thus, for every $ \tau_0 \in \{0,1,\ldots, p^{\delta}-1\}$ and $\tau_1,\tau_2 \ldots, \tau_{n-\delta} \in \{0, 1, \ldots, p-1\}$, we have
\begin{equation}\label{e33}
a_0(c_{\tau_0 + \tau_1p^{\delta} + \tau_2p^{\delta+1}+ \cdots + \tau_{n-\delta}p^{n-1}})=c_{(\tau_0+1)\pmod {p^{\delta}} + \tau_1p^{\delta} + \tau_2p^{\delta+1}+ \cdots + \tau_{n-\delta}p^{n-1}},
\end{equation}
and 
\begin{equation}\label{e34}
a_i(c_{\tau_0 + \tau_1p^{\delta} + \tau_2p^{\delta+1}+ \cdots + \tau_{n-\delta}p^{n-1}})=c_{\tau_0 + \tau_1p^{\delta} + \tau_2p^{\delta+1}+ \cdots+\tau_{i-1}p^{i+\delta-2}+((\tau_i+1)\pmod p)p^{i+\delta-1}+\tau_{i+1}p^{i+\delta}+\cdots+ \tau_{n-\delta}p^{n-1}},
\end{equation}
where Equation \eqref{e33} is obtained from Equation \eqref{e31} by choosing $j_0=\tau_1 + \tau_2p+ \cdots + \tau_{n-\delta}p^{n-\delta-1}$, $r_0=\tau_0$ and Equation \eqref{e34} comes from Equation \eqref{e32} by taking $r_i=\tau_i$ and $$j_{i}=\tau_0 + \tau_1p^{\delta} + \tau_2p^{\delta+1}+ \cdots+\tau_{i-1}p^{i+\delta-2}+\tau_{i+1}p^{i+\delta}+\cdots + \tau_{n-\delta}p^{n-1}.$$
Equation \eqref{e33} and Equation \eqref{e34} are employed  $w_1$ and $w_2$ times, respectively, to obtain
\begin{equation}\label{e35}
a_0^{w_1}(c_{\tau_0 + \tau_1p^{\delta} + \tau_2p^{\delta+1}+ \cdots + \tau_{n-\delta}p^{n-1}})=c_{(\tau_0+w_1)\pmod {p^{\delta}} + \tau_1p^{\delta} + \tau_2p^{\delta+1}+ \cdots + \tau_{n-\delta}p^{n-1}},
\end{equation}
and 
\begin{equation}\label{e36}
\begin{split}
&a_i^{w_2}(c_{\tau_0 + \tau_1p^{\delta} + \tau_2p^{\delta+1}+ \cdots + \tau_{n-\delta}p^{n-1}})\\&~=c_{\tau_0 + \tau_1p^{\delta} + \tau_2p^{\delta+1}+ \cdots+\tau_{i-1}p^{i+\delta-2}+((\tau_i+w_2)\pmod p)p^{i+\delta-1}+\tau_{i+1}p^{i+\delta}+\cdots+ \tau_{n-\delta}p^{n-1}},
\end{split}
\end{equation}
where $w_1 \in \{0,1,\ldots,p^{\delta}-1\}$ and $w_2 \in \{0,1,\ldots,p-1\}$.

Now, we aim to show that $a_ha_k(c_t) = a_ka_h(c_t)$ for any $c_t \in \mathbb{F}_q$ and $h,k \in \{0,1,\ldots,n-\delta\}$. The case where $h = k $ is trivial, so without loss of generality assume that $k > h$. It is easy to see that every $t \in \{0, 1, \ldots, q-1=p^n-1\}$ can be expressed as $t = t_0 + t_1p^{\delta} + t_2p^{\delta+1}+ \cdots + t_{n-\delta}p^{n-1}$, where $ t_0 \in \{0,1,\ldots, p^{\delta}-1\}$ and $t_1,t_2, \ldots, t_{n-\delta} \in \{0, 1, \ldots, p-1\}$. 

\textbf{Case 1:} Let $h=0$. In this case, from Equation \eqref{e33}, we have
\[
a_ka_0(c_t)=a_ka_0(c_{t_0 + t_1p^{\delta} + t_2p^{\delta+1}+ \cdots + t_{n-\delta}p^{n-1}})=a_k(c_{(t_0+1)\pmod {p^{\delta}} + t_1p^{\delta} + t_2p^{\delta+1}+ \cdots + t_{n-\delta}p^{n-1}}).
\]
Further using Equation \eqref{e34} for $i=k$, 
\begin{equation*}
\begin{split}
&a_k(c_{(t_0+1)\pmod {p^{\delta}} + t_1p^{\delta} + t_2p^{\delta+1}+ \cdots + t_{n-\delta}p^{n-1}})\\&=c_{(t_0+1)\pmod {p^{\delta}} + t_1p^{\delta} + \cdots+t_{k-1}p^{k+\delta-2}+((t_k+1)\pmod p)p^{k+\delta-1}+ t_{k+1}p^{k+\delta}+\cdots + t_{n-\delta}p^{n-1}}.
\end{split}
\end{equation*}
On the other hand, 
\begin{equation*}
\begin{split}
a_0a_k(c_t)&=a_0a_k(c_{t_0 + t_1p^{\delta} + t_2p^{\delta+1}+ \cdots + t_{n-\delta}p^{n-1}})\\&=a_0(c_{t_0 + t_1p^{\delta} + t_2p^{\delta+1}+ \cdots+t_{k-1}p^{k+\delta-2}+((t_k+1)\pmod p)p^{k+\delta-1}+t_{k+1}p^{k+\delta}+\cdots+ t_{n-\delta}p^{n-1}})\\&=c_{(t_0+1)\pmod {p^{\delta}} + t_1p^{\delta} + \cdots+t_{k-1}p^{k+\delta-2}+((t_k+1)\pmod p)p^{k+\delta-1}+ t_{k+1}p^{k+\delta}+\cdots + t_{n-\delta}p^{n-1}}=a_ka_0(c_t).
\end{split}
\end{equation*}

\textbf{Case 2:} Suppose $ h \geq 1$ . In this case from Equation \eqref{e34},
\begin{equation*}
\begin{split}
&a_ha_k(c_t)=a_ha_k(c_{t_0 + t_1p^{\delta} + \cdots + t_{n-\delta}p^{n-1}})\\&=a_h(c_{t_0 + t_1p^{\delta} + t_2p^{\delta+1}+ \cdots+t_{k-1}p^{k+\delta-2}+((t_k+1)\pmod p)p^{k+\delta-1}+t_{k+1}p^{k+\delta}+\cdots+ t_{n-\delta}p^{n-1}})\\&=c_{w},
\end{split}
\end{equation*}
where $w=t_0 + t_1p^{\delta} +  \cdots+t_{h-1}p^{h+\delta-2}+((t_h+1)\pmod p)p^{h+\delta-1}+t_{h+1}p^{h+\delta}+\cdots+ t_{k-1}p^{k+\delta-2}+((t_k+1)\pmod p)p^{k+\delta-1}+t_{k+1}p^{k+\delta}+\cdots+ t_{n-\delta}p^{n-1}$. Similarly, $a_ka_h(c_t)=c_{w}$. Hence, $a_k a_h=a_h a_k$ for all $h,k\in \{0,1,\ldots,n-\delta\}$, which proves that  $G$ is an Abelian group. Next, we shall show that $|G| = p^n = q$. Suppose that there exist $i_0,i'_0 \in \{0,1,\ldots,p^{\delta}-1\}$ and $i_{\ell}, i'_{\ell} \in \{0, 1, \ldots, p-1\}$, where $1 \leq \ell \leq n-\delta$ such that 
\[
 a_0^{i_0} a_1^{i_1} \cdots a_{n-\delta}^{i_{n-\delta}} = a_0^{i'_0} a_1^{i'_1} \cdots a_{n-\delta}^{i'_{n-\delta}}.
\]
Then
\[
 a_0^{i_0} a_1^{i_1} \cdots a_{n-\delta}^{i_{n-\delta}} (c_{t_0 + t_1p^{\delta} + t_2p^{\delta+1}+ \cdots + t_{n-\delta}p^{n-1}})= a_0^{i'_0} a_1^{i'_1} \cdots a_{n-\delta}^{i'_{n-\delta}}(c_{t_0 + t_1p^{\delta} + t_2p^{\delta+1}+ \cdots + t_{n-\delta}p^{n-1}}).
\]
Using Equations \eqref{e35} and \eqref{e36}, we get
\begin{equation*}
\begin{split}
&c_{(t_0+i_0) \pmod {p^{\delta}}  + ((t_1+i_1)\pmod p) p^{\delta} + ((t_2+i_2)\pmod p)p^{\delta+1}+ \cdots + ((t_{n-\delta}+i_{n-\delta})\pmod p)p^{n-1}}\\&=c_{(t_0+i'_0) \pmod {p^{\delta}}+ ((t_1+i'_1)\pmod p)p^{\delta} + ((t_2+i'_2)\pmod p)p^{\delta+1}+ \cdots + ((t_{n-\delta}+i'_{n-\delta})\pmod p)p^{n-1}}.
\end{split}
\end{equation*}
This gives that 
\[
\begin{cases}
t_0+i_0  \equiv t_0+i'_0  \pmod{p^{\delta}}, \\
t_1+i_1  \equiv  t_1+i'_1 \pmod{p}, \\
t_2+i_2  \equiv t_2+i'_2  \pmod{p}, \\
\vdots \\
t_{n-\delta}+i_{n-\delta} \equiv t_{n-\delta}+i'_{n-\delta} \pmod{p}
\end{cases} \text{ or, equivalently,          }\quad \begin{cases}
i_0  \equiv i'_0  \pmod{p^{\delta}}, \\
i_1  \equiv  i'_1 \pmod{p}, \\
i_2  \equiv i'_2  \pmod{p}, \\
\vdots \\
i_{n-\delta} \equiv i'_{n-\delta} \pmod{p}.
\end{cases}
\]
Since $i_0,i'_0 \in \{0,1,\ldots,p^{\delta}-1\}$ and $i_{\ell}, i'_{\ell} \in \{0, 1, \ldots, p-1\}$ for all $1 \leq \ell \leq n-\delta$, therefore, we have $i_0=i'_0$ and $i_{\ell}= i'_{\ell}$ for all $1 \leq \ell \leq n-\delta$. Thus, $|G|=q$.

Now, it only remains to show that no permutation, except the identity permutation, of $G$ has a fixed point in $\F_q$. Let 
\[
a = a_{0}^{i_{0}} a_{1}^{i_{1}} a_2^{i_{2}}\cdots a_{n-\delta}^{i_{n-\delta}} \neq I
\] 
be an permutation of $G$, where $i_0 \in \{0,1,\ldots,p^{\delta}-1\}$, $i_{1}, i_{2},\ldots,i_{n-\delta} \in \{0, 1, \ldots, p-1\}$ and $I$ is the identity permutation of $\mathfrak{S}_q$.  We shall show that $a$ does not fix any element of $\mathbb{F}_q$.  On the contrary, let us assume that $a$ fixes $c_t$ , where $t= t_0 + t_1p^{\delta} + t_2p^{\delta+1}+ \cdots + t_{n-\delta}p^{n-1}$ for some $t_0 \in \{0,1,\ldots,p^{\delta}-1\}$ and $t_1,t_2,\ldots,t_{n-\delta} \in \{0,1,\ldots,p-1\}$, that is,
\[
a_{0}^{i_{0}} a_{1}^{i_{1}} a_2^{i_{2}}\cdots a_{n-\delta}^{i_{n-\delta}}(c_{t_0 + t_1p^{\delta} + t_2p^{\delta+1}+ \cdots + t_{n-\delta}p^{n-1}})=c_{t_0 + t_1p^{\delta} + t_2p^{\delta+1}+ \cdots + t_{n-\delta}p^{n-1}}.
\]
Now, by applying Equations \eqref{e35} and \eqref{e36}, we obtain
\begin{equation*}
\begin{split}
&c_{(t_0+i_0) \pmod {p^{\delta}} + ((t_1+i_1)\pmod p)p^{\delta} + ((t_2+i_2)\pmod p)p^{\delta+1}+ \cdots + ((t_{n-\delta}+i_{n-\delta})\pmod p)p^{n-1}}\\&= c_{t_0 + t_1p^{\delta} + t_2p^{\delta+1}+ \cdots + t_{n-\delta}p^{n-1}},
\end{split}
\end{equation*}
which gives that $i_k=0$ for $0\leq k \leq n-\delta$, i.e., $a=I$. This is a contradiction, which completes the proof.
\end{proof}

Now, we will establish that our permutation group polynomials are not equivalent with known ones. We classify them with respect to the following equivalence relation that is used in \cite{JJ_2023}.

\begin{defn}\label{D31}\cite[Definition 13]{JJ_2023}
Two permutation polynomial tuples $\Omega$ and $\Gamma$ are said to be equivalent if there exist
$\sigma, \lambda \in \mathfrak{S}_q$ such that $\sigma \Omega \lambda = \Gamma$.
Analogously, we say that two local permutation polynomials $f$ and $g$ are equivalent if the
corresponding permutation polynomial tuples $\underline{\beta}_f=(\beta_0, \ldots, \beta_{q-1})$ and $\underline{\gamma}_g=(\gamma_0, \ldots, \gamma_{q-1})$ are equivalent.
\end{defn}

To the best of our knowledge, there are only four classes of permutation group polynomials \cite{JJ_2023, HK}, and their corresponding groups have at most two generators. 

\begin{rmk}\label{R31}
In the proof of \cite[Theorem 3.18]{HK}, it is shown that if two permutation group polynomials are equivalent to each other, then their corresponding groups are conjugates to each other in \( \mathfrak{S}_q \).
\end{rmk}

\begin{prop}\label{P31}
Let $q = p^n$, where $p$ is a prime number and $n   \ge 4 $ is a positive integer. Then the permutation group polynomial $f \in \mathbb{F}_q[X_1, X_2]$ corresponding to $G$ defined in Theorem \ref{T31} is not equivalent to the known families of permutation group polynomials.
\end{prop}

\begin{proof}
The proof follows from Remark \ref{R31}, as $G$ has more than two generators for $n \geq 4$.
\end{proof}

In the following theorem, we provide an explicit construction of the companions for the permutation group polynomials determined in Theorem \ref{T31}.  Recall that every $t \in \{0, 1, \ldots, q-1\}$ can be expressed as \( t = t_0 + t_1p^{\delta} + t_2p^{\delta+1} + \cdots + t_{n-\delta}p^{n-1} \), where $t_0 \in \{0,1,\ldots,p^{\delta}-1\}$, \( t_1, t_2 \ldots, t_{n-\delta} \in \{0, 1, \ldots, p-1\} \) and $\delta \in \{1,2\}$.
\begin{thm}\label{T33}
Let $q = p^n$, $n\ge 2$ is a positive integer and $p$ is a prime number, and $f \in \F_q[X_1, X_2]$ be a permutation group polynomial associated with the group $G$ constructed in Theorem \ref{T31}. Let $\underline{\beta}_{f} = (\beta_0, \beta_1, \ldots, \beta_{q-1})$ be the permutation polynomial tuple corresponding to $f$, where $\beta_{t} = a_0^{t_0}a_1^{t_1} \cdots a_{n-\delta}^{t_{n-\delta}}$ for $
t = t_0 + t_1 p^{\delta} + t_2 p^{\delta+1} + \cdots + t_{n-\delta} p^{n-1}$ with $t_0 \in \{0,1,\ldots,p^{\delta}-1\}$ and $t_1,t_2,\ldots,t_{n-\delta} \in \{0,1,\ldots,p-1\}$. Moreover, define  $h :\F_q \rightarrow \F_q$ as follows
\[
h(c_{t})=\displaystyle
\begin{cases}
a_0^{t_0}a_1^{t_1}a_2^{t_2}\cdots a_{n-\delta}^{t_{n-\delta}}(c_{t}) & \text{ if } p \text{ is odd},\\[0.3cm]
a_0^{t_{n-1}}a_1^{t_0 + t_{n-1}}a_2^{t_1}a_3^{t_2}\cdots a_{n-1}^{t_{n-2}}(c_{t}) & \text{ if } p=2, \delta=1,\\[0.3cm]
a_0^{t_{n-2} + 2t_{n-3}} a_1^{t_{00} + S_1^t-t_1} a_2^{t_{01} +S_2^t-t_2} a_3^{t_1} a_4^{t_2} \cdots a_{n-2}^{t_{n-4}}(c_{t}) & \text{ if } p=2, \delta=2, 
\end{cases}
\]
 where for the case $p=2$ and $\delta=2$ above, $n \geq 5$, $t_0 = t_{00} + 2t_{01}$ for some unique $t_{00}, t_{01} \in \mathbb{Z}_2$,  and \[
S_1^{t}=\displaystyle \sum_{k=1}^{\left \lfloor \frac{n-3}{2} \right \rfloor}(t_{2k-1} + t_{2k+1}) \pmod 2, \quad S_2^{t}=\displaystyle \sum_{k=1}^{\left \lfloor \frac{n-4}{2} \right \rfloor}(t_{2k} + t_{2k+2}) \pmod 2.
\]
Then the bivariate local permutation polynomial $g(X_1,X_2)$ corresponding to the permutation polynomial tuple $(h\beta_0, h\beta_1, \ldots, h\beta_{q-1})$ is a companion of $f$.
\end{thm}

\begin{proof}
From Lemma \ref{L21}, to prove $g(X_1,X_2)$ is a companion of $f(X_1,X_2)$, it suffices to show that  $h \in \F_q[X]$ is a permutation polynomial of $\F_q$ and it intersects $f$ simply. We split our proof in the following cases.

\textbf{Case 1:} Let $p$ be an odd prime. In this case, we have
\[
h(c_{t_0+t_1p^{\delta}+t_2p^{\delta+1}+\cdots+t_{n-\delta}p^{n-1}})=a_0^{t_0}a_1^{t_1}a_2^{t_2}\cdots a_{n-\delta}^{t_{n-\delta}}(c_{t_0+t_1p+t_2p^2+\cdots+t_{n-\delta}p^{n-1}}),
\]
where  $t_0\in \{0,1,\ldots,p^{\delta}-1\}$ and $t_{\ell} \in \{0,1,\ldots,p-1\}$ for all $1 \leq \ell \leq  n-\delta$. By definition of $h$, it is clear that $h$ intersects $f$ simply. Now, we shall show that $h$ is a permutation. Suppose, to the contrary, that
\[
h(c_{t_0+t_1p^{\delta}+t_2p^{\delta+1}+\cdots+t_{n-\delta}p^{n-1}}) = h(c_{t'_0+t'_1p^{\delta}+t'_2p^{\delta+1}+\cdots+t'_{n-\delta}p^{n-1}}),
\]
for some \( (t_0,t_1,\ldots,t_{n-\delta}) \neq (t'_0,t'_1,\ldots,t'_{n-\delta}) \), where $t_0,t'_0 \in \{0,1,\ldots,p^{\delta}-1\}$ and \( t_{\ell}, t'_{\ell} \in \{0,1,\ldots,p-1\} \) for all \( 1 \leq \ell \leq  n-\delta \).
This implies that
\[
a_0^{t_0}a_1^{t_1}a_2^{t_2}\cdots a_{n-\delta}^{t_{n-\delta}}(c_{t_0+t_1p^{\delta}+t_2p^{\delta+1}+\cdots+t_{n-\delta}p^{n-1}}) =
a_0^{t'_0}a_1^{t'_1}a_2^{t'_2}\cdots a_{n-\delta}^{t'_{n-\delta}}(c_{t'_0+t'_1p^{\delta}+t'_2p^{\delta+1}+\cdots+t'_{n-\delta}p^{n-1}}).
\]
From Equations \eqref{e35} and \eqref{e36}, we obtain
\begin{equation*}
\begin{split}
&c_{2t_0 \pmod {p^{\delta}} + (2t_1 \pmod p) p^{\delta} + (2t_2 \pmod p)p^{\delta+1} + \cdots + (2t_{n-\delta} \pmod p)p^{n-1}}\\
 & =c_{2t'_0 \pmod {p^{\delta}} + (2t'_1 \pmod p) p^{\delta} + (2t'_2 \pmod p)p^{\delta+1} + \cdots + (2t'_{n-\delta} \pmod p)p^{n-1}}.
\end{split}
\end{equation*}
The above equality implies that $2t_{0} \equiv 2t'_{0} \pmod{p^{\delta}}$ and $2t_{\ell} \equiv 2t'_{\ell} \pmod{p}$ for all $1 \leq  \ell \leq  n-\delta$. Since $t_0,t'_0 \in \{0,1,\ldots,p^{\delta}-1\}$, \( t_{\ell}, t'_{\ell} \in \{0,1,\ldots,p-1\} \) for $ 1\leq \ell \leq n-\delta$ and \( \gcd(2, p) = 1 \), we deduce that $t_0=t'_0$ and \( t_{\ell} = t'_{\ell} \) for all \( \ell \). Hence, \( (t_0,t_1,\ldots,t_{n-\delta}) = (t'_0,t'_1,\ldots,t'_{n-\delta}) \), which is a contradiction. Therefore, $h$ is a permutation and $g(X_1,X_2)$ is a companion of $f$.

\textbf{Case 2:} For $p = 2$, we consider two subcases, corresponding to the values $\delta=1$ and $\delta=2$. 

\textbf{Subcase 2.1 :} Let $\delta=1$. In this case, we can write every $t \in \{0,1,\ldots,p^n-1\}$ as $t=t_0 + 2t_1 + 2^2t_2 + \cdots + 2^{n-1}t_{n-1}$ for some $t_0,t_1,\ldots,t_{n-1} \in \{0,1\}$ and by definition of $h$, we have
\[
h(c_{t_0 + 2t_1 + 2^2t_2 + \cdots + 2^{n-1}t_{n-1}}) = a_0^{t_{n-1}}a_1^{t_0 + t_{n-1}}a_2^{t_1}a_3^{t_2}\cdots a_{n-1}^{t_{n-2}}(c_{t_0 + 2t_1 + 2^2t_2 + \cdots + 2^{n-1}t_{n-1}}).
\]
If \( (t_0, t_1, t_2, \ldots, t_{n-1}) \neq (t'_0, t'_1, t'_2, \ldots, t'_{n-1}) \in \mathbb{Z}_2^n \), where \(\mathbb{Z}_2 =\{0,1\} \), then 
\[ (t_{n-1}, t_0 + t_{n-1}, t_1, t_2, \ldots, t_{n-2}) \neq (t'_{n-1}, t'_0 + t'_{n-1}, t'_1, t'_2, \ldots, t'_{n-2}) .\] Thus, \( h \) intersects \( f \) simply. Now, we aim to show that \( h \) is a permutation. Suppose \( (t_0, t_1, t_2, \ldots, t_{n-1}) \neq (t'_0, t'_1, t'_2, \ldots, t'_{n-1}) \in \mathbb{Z}_2^n \) such that 
\[
h(c_{t_0 + 2t_1 + 2^2t_2 + \cdots + 2^{n-1}t_{n-1}}) = h(c_{t'_0 + 2t'_1 + 2^2t'_2 + \cdots + 2^{n-1}t'_{n-1}}).
\]
From the Equations \eqref{e35} and \eqref{e36} for $\delta=1$, we have
\begin{equation*}
\begin{split}
&c_{(t_0 + t_{n-1}) \pmod 2 + 2((t_1 + t_0 + t_{n-1}) \pmod 2) + 2^2((t_2 + t_1) \pmod 2) +  \cdots + 2^{n-1}((t_{n-1} + t_{n-2}) \pmod 2)} \\
&= c_{(t'_0 + t'_{n-1}) \pmod 2 + 2((t'_1 + t'_0 + t'_{n-1}) \pmod 2) + 2^2((t'_2 + t'_1) \pmod 2) +  \cdots + 2^{n-1}((t'_{n-1} + t'_{n-2}) \pmod 2)}.
\end{split}
\end{equation*}
This implies
\[
\begin{cases}
t_0 + t_{n-1} \equiv t'_0 + t'_{n-1} \pmod{2}, \\
t_1 + t_0 + t_{n-1} \equiv t'_1 + t'_0 + t'_{n-1} \pmod{2}, \\
t_2 + t_1 \equiv t'_2 + t'_1 \pmod{2}, \\
t_3 + t_2 \equiv t'_3 + t'_2 \pmod{2}, \\
\vdots \\
t_{n-1} + t_{n-2} \equiv t'_{n-1} + t'_{n-2} \pmod{2}.
\end{cases}
\]

Adding the first two equations, we obtain \( t_1 \equiv t'_1 \pmod{2} \), which implies \( t_1 = t'_1 \). Substituting \( t_1 = t'_1 \) into the third equation, we get \( t_2 = t'_2 \). Continuing this process, we find \( t_3 = t'_3, t_4 = t'_4, \ldots, t_{n-1} = t'_{n-1}\). Finally, substituting \( t_{n-1} = t'_{n-1} \) into the first equation, we get \( t_0 = t'_0 \). Thus, \( c_{t_0 + 2t_1 + 2^2t_2 + \cdots + 2^{n-1}t_{n-1}} = c_{t'_0 + 2t'_1 + 2^2t'_2 + \cdots + 2^{n-1}t'_{n-1}} \), which shows that \( h \) is a permutation when \( p = 2 \) and $\delta=1$.

\textbf{Subcase 2.2 :} Now we assume that $\delta=2$. Then every \( t \in \{0, 1, \ldots, q-1\} \) can be written as 
\[
t = t_0 + 2^2t_1  + 2^3t_2 + \cdots + 2^{n-1}t_{n-2},
\]
where $ t_0 \in \{0, 1,2,3\}$ and $t_1, \ldots, t_{n-2} \in \{0, 1\}$. Note that \( t_0 \) can be written as \( t_0 = t_{00} + 2t_{01} \) for some unique \( t_{00}, t_{01} \in \mathbb{Z}_2 \) and in this case, 
\[
h(c_{t}) = a_0^{t_{n-2} + 2t_{n-3}} a_1^{t_{00} + S_1^t-t_1} a_2^{t_{01} +S_2^t-t_2} a_3^{t_1} a_4^{t_2} \cdots a_{n-2}^{t_{n-4}}(c_{t}),
\]
where $
t = t_{00} + 2t_{01} + 2^2t_1 + 2^3t_2 + \cdots + 2^{n-2}t_{n-3} + 2^{n-1}t_{n-2}$,
\[
S_1^{t}=\displaystyle \sum_{k=1}^{\left \lfloor \frac{n-3}{2} \right \rfloor}(t_{2k-1} + t_{2k+1}) \pmod 2, \quad S_2^{t}=\displaystyle \sum_{k=1}^{\left \lfloor \frac{n-4}{2} \right \rfloor}(t_{2k} + t_{2k+2}) \pmod 2.
\]
It is straightforward to verify that for distinct choices of the tuple \( (t_{00}, t_{01}, t_1, t_2, \ldots, t_{n-2}) \in \mathbb{Z}_2^n \), the tuple \( (t_{n-2} + 2t_{n-3}, t_{00} + S^t_1-t_1, t_{01} +S^t_2-t_2, t_1, t_2, \ldots, t_{n-4}) \in \mathbb{Z}_4 \times \mathbb{Z}_2 \times \mathbb{Z}_2 \times \cdots \times \mathbb{Z}_2 \) will also be distinct. Therefore, \( h \) intersects \( f \) simply. We now show that \( h \) is a permutation. Assume, for the sake of contradiction, that there are  $t_{00} + 2t_{01} + 2^2t_1 + 2^3t_2 + \cdots + 2^{n-2}t_{n-3} + 2^{n-1}t_{n-2}=t \neq  t' = t'_{00} + 2t'_{01} + 2^2t'_1 + 2^3t'_2 + \cdots + 2^{n-2}t'_{n-3} + 2^{n-1}t'_{n-2}$ such that
\[
h(c_{t_{00} + 2t_{01} + 2^2t_1 + 2^3t_2 + \cdots + 2^{n-2}t_{n-3} + 2^{n-1}t_{n-2}}) = h(c_{t'_{00} + 2t'_{01} + 2^2t'_1 + 2^3t'_2 + \cdots + 2^{n-2}t'_{n-3} + 2^{n-1}t'_{n-2}}).
\]
Using the definition of \( h \) and Equations \eqref{e35},  \eqref{e36} for $\delta=2$, we obtain
\[
c_{s_1} = c_{s_2},
\]
where
\begin{align*}
s_1 = &(t_{00} + 2t_{01} + t_{n-2} + 2t_{n-3}) \pmod 4 +((t_{00}+S^t_1) \pmod 2)2^2+((t_{01}+S^t_2) \pmod 2)2^3\\
& + \sum_{k=1}^{\left \lfloor \frac{n-3}{2} \right \rfloor} ((t_{2k-1} + t_{2k+1}) \pmod 2)2^{2k+2} +\sum_{k=1}^{\left \lfloor \frac{n-4}{2} \right \rfloor} ((t_{2k} + t_{2k+2}) \pmod 2)2^{2k+3},
\end{align*}
and
\begin{align*}
s_2 = &(t'_{00} + 2t'_{01} + t'_{n-2} + 2t'_{n-3}) \pmod 4 +((t'_{00}+S_1^{t'}) \pmod 2)2^2+((t'_{01}+S_2^{t'}) \pmod 2)2^3\\
& + \sum_{k=1}^{\left \lfloor \frac{n-3}{2} \right \rfloor} ((t'_{2k-1} + t'_{2k+1}) \pmod 2)2^{2k+2} +\sum_{k=1}^{\left \lfloor \frac{n-4}{2} \right \rfloor} ((t'_{2k} + t'_{2k+2}) \pmod 2)2^{2k+3}.
\end{align*}
This implies the following system of congruences
\[
\begin{cases}
t_{00} + 2t_{01} + t_{n-2} + 2t_{n-3} \equiv t'_{00} + 2t'_{01} + t'_{n-2} + 2t'_{n-3} \pmod 4, 
\\t_{00} +S_1^{t} \equiv t'_{00} + S_1^{t'} \pmod 2, \\
t_{01} + S_2^{t} \equiv  t'_{01} + S_2^{t'} \pmod 2, \\
t_{2k-1} + t_{2k+1} \equiv t'_{2k-1} + t'_{2k+1} \pmod 2, \text{ for }1\leq k \leq \left \lfloor \frac{n-3}{2} \right \rfloor  \\
t_{2k} + t_{2k+2} \equiv t'_{2k} + t'_{2k+2} \pmod 2, \text{ for }1 \leq k \leq \left \lfloor \frac{n-4}{2} \right \rfloor.
\end{cases}
\]
Next, by summing over $k = 1$ to $\left \lfloor \frac{n-3}{2} \right \rfloor$ in the 4th equation and $k = 1$ to $\left \lfloor \frac{n-4}{2} \right \rfloor$ in the 5th equation of the above system, we get 
\[
S_1^{t}=\displaystyle \sum_{k=1}^{\left \lfloor \frac{n-3}{2} \right \rfloor}(t_{2k-1} + t_{2k+1}) \pmod 2 \equiv \sum_{k=1}^{\left \lfloor \frac{n-3}{2} \right \rfloor}(t'_{2k-1} + t'_{2k+1}) \pmod 2=S_1^{t'},
\]
and
\[
\displaystyle S_2^{t}=\sum_{k=1}^{\left \lfloor \frac{n-4}{2} \right \rfloor}(t_{2k} + t_{2k+2}) \equiv \sum_{k=1}^{\left \lfloor \frac{n-4}{2} \right \rfloor}(t'_{2k} + t'_{2k+2}) \pmod 2=S_2^{t'}.
\]
Using these values in the second and third equations of the above system, we obtain \( t_{00} = t'_{00} \) and \( t_{01} = t'_{01} \), as \( t_{00}, t'_{00}, t_{01}, t'_{01} \in \mathbb{Z}_2 \). Substituting these values into the first equation, we have 
\[
t_{n-2} \equiv t'_{n-2} \pmod 2,
\]
which gives \( t_{n-2} = t'_{n-2} \). Additionally, 
\[
2t_{n-3} \equiv 2t'_{n-3} \pmod 4,
\]
implies \( t_{n-3} = t'_{n-3} \).

If $n$ is even, then substitute  \( t_{n-3} = t'_{n-3} \) and \( t_{n-2} = t'_{n-2} \) into the 4th and 5th equations for \( k =\left \lfloor \frac{n-4}{2} \right \rfloor=\left \lfloor \frac{n-3}{2} \right \rfloor= \frac{n-4}{2} \), we find 
\[
t_{n-5} = t'_{n-5}, \quad \text{and} \quad t_{n-4} = t'_{n-4}.
\]

Continuing this process and using \( k = \frac{n-2\ell}{2} \) in the 4th and 5th equations, we deduce 
\[
t_{n-(2\ell+1)} = t'_{n-(2\ell+1)} \quad\text{and} \quad t_{n-2\ell} = t'_{n-2\ell},
\]
where \( 3 \leq \ell \leq \frac{n-2}{2} \). Therefore, 
\[
(t_{00}, t_{01}, t_1, t_2, \ldots, t_{n-2}) = (t'_{00}, t'_{01}, t'_1, t'_2, \ldots, t'_{n-2}),
\]
when $n$ is even. Similarly one can obtain that \(
(t_{00}, t_{01}, t_1, t_2, \ldots, t_{n-2}) = (t'_{00}, t'_{01}, t'_1, t'_2, \ldots, t'_{n-2})
\) when $n$ is odd. This implies that
\( h \) is a permutation, which completes the proof.

\end{proof}

\section{Enumeration of certain permutation group polynomials}\label{S4}
In \cite{JJ_2023}, Gutiérrez and Urroz posed the problem of enumerating $e$-Klenian polynomials, that is, determining the number of permutation group polynomials whose associated groups are subgroups of $\mathfrak{S}_q$ with the structure described in \cite[Corollary 17]{JJ_2023}. This problem highlights the broader challenge of counting permutation group polynomials, which has deep connections with group theory and combinatorics.
Motivated by this question, we consider a more general enumeration problem: given a subgroup $G < \mathfrak{S}_q$, determine the number of permutation group polynomials whose associated groups are subgroups of $\mathfrak{S}_q$ and isomorphic to $G$. Equivalently, we seek to count all permutation group polynomials whose associated groups are subgroups  of $\mathfrak{S}_q$ of the form $G$. For convenience, we shall henceforth refer that two permutation group polynomials are of the same form if their associated groups, as subgroups of $\mathfrak{S}_q$, are isomorphic. In this section, we enumerate all permutation group polynomials of the form determined in Section~\ref{S3} and provide the number of permutation group polynomials that are equivalent to this family. We then solve the problem of counting of $e$-Klenian polynomials over $\mathbb{F}_q$ for $e \geq 1$. At the end of this section, we give the number of permutation group polynomials equivalent to $e$-Klenian polynomials.

To enumerate all the permutation group polynomials of the form presented in Theorem \ref{T31}, we first prove the following lemmas.

 \begin{lem}\label{conjugate_lemma}
Let $g(X_1,X_2) \in \F_q[X_1,X_2]$ be a permutation group polynomial and its corresponding group $H < \mathfrak{S}_q$ be isomorphic to $G$, where $G$ is defined in Theorem \ref{T31}. Then there exists $\sigma \in \mathfrak{S}_q$ such that $G=\sigma H \sigma^{-1}$.
\end{lem}
\begin{proof}
By Theorem \ref{T31}, $G \cong \mathbb{Z}_{p^{\delta}} \times \mathbb{Z}_p \times \mathbb{Z}_p \times \cdots \times \mathbb{Z}_p$, therefore, $H\cong \mathbb{Z}_{p^{\delta}} \times \mathbb{Z}_p \times \mathbb{Z}_p \times \cdots \times \mathbb{Z}_p$ as $H\cong G$. Moreover, $H$ is a subgroup of $\mathfrak{S}_q$ associated with a permutation group polynomial $g(X_1,X_2)$, it follows that $H=\langle b_0,b_1,\ldots,b_{n-\delta} \rangle$, where $|b_0|=p^{\delta}$, $|b_i|=p$ for all $1 \le i \le n-\delta$ and every non identity permutation of $H$ has no fixed point in $\F_q$. Moreover, due to \cite[Lemma 15]{JJ_2023} and the fact that $|b_0|=p^{\delta}$, $b_0$ is a product of $p^{n-\delta}$ disjoint cycles each of length $p^{\delta}$. Similarly $b_i$, $1 \le i \le n-\delta$, is a  product of $p^{n-1}$ disjoint cycles each of length $p$. As $a_0 \in G$ and $b_0 \in H$ have same cyclic structure, there exists $\sigma_{0} \in \mathfrak{S}_q$ such that  $\sigma_{0} b_0 \sigma_{0}^{-1}=a_0$ and hence 
\[
\sigma_{0} H \sigma_{0}^{-1}=\langle a_0,b'_1,\ldots,b'_{n-\delta} \rangle,
\] 
where $b'_i=\sigma_{0}b_i \sigma_{0}^{-1}.$

To prove the result, it suffices to show that $G$ is conjugate to the subgroup $\sigma_{0} H \sigma_{0}^{-1}$. We proceed by induction on $n - \delta$. First, we verify the base case $ n - \delta = 1$.
In this case, we shall prove that there exists an element \( \sigma_{1} \in \mathfrak{S}_q \) such that
\[
\sigma_{1}\,\sigma_{0} H \sigma_{0}^{-1}\,\sigma_{1}^{-1}
= \langle a_0, a_1 \rangle.
\]

Since $\sigma_{0} H \sigma_{0}^{-1}=\langle a_0,b'_1\rangle$ is a conjugate of $H$, the subgroup $\sigma_{0} H \sigma_{0}^{-1}$ satisfies the following properties: 
\begin{enumerate}[(a)]
\item every non identity permutation of $\langle a_0,b'_1\rangle$ has no fixed point in $\F_q$.
\item  $b'_1a_0=a_0 b'_1$, i.e., $b'_1a_0 {b'_1}^{-1}=a_0$.
\item $b'_1$ is a  product of $p^{n-1}$ disjoint cycles each of length $p$.
\end{enumerate}
Using the above properties, we now analyze the cyclic structure of $b'_1$. 
Observe that $b'_1(c_0)\neq c_t$ for any $0\le t\le p^{\delta}-1$. 
Indeed, if $b'_1(c_0)=c_t$ for some $t$, then by Equation \eqref{e35},
\(
a_0^{-t} b'_1(c_0)=c_0,
\)
which implies that the non identity permutation $a_0^{-t} b'_1 \in \sigma_{0} H \sigma_{0}^{-1}$ fixes $c_0$. This contradicts property~(a) of  $\sigma_{0} H \sigma_{0}^{-1}$. Therefore, 
$$b'_1(c_0)=c_{t_{1,1}p^{\delta}+t_{0,1}},$$ where $0 \le t_{0,1} \le p^{\delta}-1$ and $1\le t_{1,1} \le p^{n-\delta}-1$. Next, $b'_1(c_{t_{1,1}p^{\delta}+t_{0,1}})\neq c_{t_{1,1}p^{\delta}+t}$ for any $0\leq t \leq p^{\delta}-1$. If $b'_1(c_{t_{1,1}p^{\delta}+t_{0,1}})= c_{t_{1,1}p^{\delta}+t}$ for some $0\leq t \leq p^{\delta}-1$, then $a^{t_{0,1}-t}_0b'_1(c_{t_{1,1}p^{\delta}+t_{0,1}})=c_{t_{1,1}p^{\delta}+t_{0,1}}$, which is not possible. Similarly, $b'_1(c_{t_{1,1}p^{\delta}+t_{0,1}})\neq c_t$, otherwise $a_0^{-t} {b'_1}^{2}(c_0)=c_0$. Thus,
$$b'_1\left(c_{t_{1,1}p^{\delta}+t_{0,1}}\right)=c_{t_{1,2}p^{\delta}+t_{0,2}},$$ 
where $0 \le t_{0,2} \le p^{\delta}-1$,$0 \le t_{1,2} \le p^{n-\delta}-1$ and $t_{1,2}\not \in \{0, t_{1,1}\}.$ Continuing in this way, we get 
$$b'_1\left(c_{t_{1,p-2}p^{\delta}+t_{0,p-2}}\right)=c_{t_{1,p-1}p^{\delta}+t_{0,p-1}},$$
 where $0 \le t_{0,p-1} \le p^{\delta}-1$,$0 \le t_{1,p-1} \le p^{n-\delta}-1$ and $t_{1,p-1}\not \in \{0, t_{1,1},\ldots, t_{1,p-2}\}.$ Also, each cycle of $b'_1$ has length $p$, it follows that 
 $$b'_1\left(c_{t_{1,p-1}p^{\delta}+t_{0,p-1}}\right)=c_0.$$
From the identity $b'_1 a_0 {b'_1}^{-1}=a_0$ in (b), we have that $b'_1 a_0^{\,t}=a_0^{\,t} b'_1$ for all 
$1\le t\le p^{\delta}-1$. Hence for $0\le u \le p-2$, we obtain
\[
b'_1 \left(a_0^{t}\left(c_{t_{1,u}p^{\delta}+t_{0,u}}\right)\right)= a_0^{t}\left(b'_1\left(c_{t_{1,u}p^{\delta}+t_{0,u}}\right)\right),
\]
where $t_{1,0}=0$ and $t_{0,0}=0$.
Equation \eqref{e35}, together with above discussion, yields that
\begin{align*}
b'_1\left(c_{t_{1,u}p^{\delta}+(t_{0,u}+t)\pmod {p^{\delta}}}\right)=
b'_1 \left(a_0^{t}\left(c_{t_{1,u}p^{\delta}+t_{0,u}}\right)\right)
&= a_0^{t}\left(b'_1\left(c_{t_{1,u}p^{\delta}+t_{0,u}}\right)\right)\\
 &=c_{t_{1,u+1}p^{\delta}+(t_{0,u+1}+t)\pmod {p^{\delta}}}.
\end{align*}
Similarly, for $u=p-1$, we have
\[
b'_1\left(c_{t_{1,p-1}p^{\delta}+(t_{0,p-1}+t)\pmod {p^{\delta}}}\right)
= c_{t_{1,0}p^{\delta}+(t_{0,0}+t)\pmod {p^{\delta}}}=c_t.
\]
Now, we start with new $t_{1,p} \not \in \{0,t_{1,1},t_{1,2},\ldots,t_{1,p-1}\}$ and apply similar argument as above to get 
\[
b'_1\left(c_{t_{1,u'}p^{\delta}+(t_{0,u'}+t)\pmod {p^{\delta}}}\right)
= c_{t_{1,u'+1}p^{\delta}+(t_{0,u'+1}+t)\pmod {p^{\delta}}}
\]
for $p \le u' \le 2p-2$ and 
\[
b'_1\left(c_{t_{1,2p-1}p^{\delta}+(t_{0,2p-1}+t)\pmod {p^{\delta}}}\right)
= c_{t_{1,p}p^{\delta}+(t_{0,p}+t)\pmod {p^{\delta}}},
\]
where $0 \le t \le p^{\delta}-1$ and all indices $1 \le t_{1,p+i} \le p^{n-\delta}-1$ are pairwise distinct for $0 \le i \le p-1$. Further we have  
$$\{t_{1,0},t_{1,1},\ldots,t_{1,p-1}\}\cap\{t_{1,p},t_{1,p+1},\ldots,t_{1,2p-1}\}=\emptyset$$ 
as $b'_1$ is a permutation. Continuing in this way, we obtain
\begin{equation*}
\begin{split}
b'_1=\displaystyle \prod_{m=0}^{p^{n-\delta-1}-1}\prod_{t=0}^{p^{\delta}-1}&\left(c_{t_{1,mp}p^{\delta}+(t_{0,mp}+t)\pmod {p^\delta}},c_{t_{1,mp+1}p^{\delta}+(t_{0,mp+1}+t)\pmod {p^\delta}},\ldots,\right.\\& \left.~~~c_{t_{1,mp+p-1}p^{\delta}+(t_{0,mp+p-1}+t)\pmod {p^\delta}}\right),
\end{split}
\end{equation*}
where $0 \le t_{1,mp+i} \le p^{n-\delta}-1$ are distinct and $0 \le t_{0,mp+i} \le p^{\delta}-1$.
We now define map $\sigma_1 \in \mathfrak{S}_q$ as 
\[
\sigma_1\left(c_{t_{1,mp+i}p^{\delta}+(t_{0,mp+i}+t)\pmod {p^\delta}}\right)=c_{(mp+i)p^{\delta}+t}
\]
and observe that  
\[
\sigma_1b'_1\sigma_1^{-1}=\displaystyle \prod_{m=0}^{p^{n-\delta-1}-1}\prod_{t=0}^{p^{\delta}-1}\left(c_{(mp) p^{\delta}+t},c_{(mp+1)p^{\delta}+t},\ldots,c_{(mp+p-1)p^{\delta}+t}\right)=a_1.
\]
Moreover, by Equation \eqref{e31}, $ \sigma_1a_0\sigma_1^{-1}=a_0$. Therefore, the result holds for \( n - \delta = 1 \).
Now assume that the result is true for \( n - \delta = k \), and we shall prove it for \( n - \delta = k + 1 \). For \( n - \delta = k + 1 \),
\[
\sigma_0H\sigma_0^{-1}=\langle a_0,b'_1,b'_2,\ldots,b'_{k+1}\rangle \quad \text{ and } \quad G=\langle a_0,a_1,a_2,\ldots,a_{k+1}\rangle.
\]
By mathematical induction, there exists $\sigma_k \in \mathfrak{S}_q$ such that 
\[
\sigma_k\langle a_0,b'_1,b'_2,\ldots,b'_{k}\rangle\sigma_k^{-1}=\langle a_0,a_1,a_2,\ldots,a_{k}\rangle \text{ and hence}
\]
\[
\sigma_k\sigma_0H\sigma_0^{-1}\sigma_k^{-1}= \langle a_0,a_1,a_2,\ldots,a_{k},b''_{k+1}\rangle
\]
for some $b''_{k+1} \in \mathfrak{S}_q$. By similar arguments as in the case $n-\delta=1$, we get the cyclic structure of $b''_{k+1}$ as follows 
\begin{equation*}
\begin{split}
b''_{k+1}=\displaystyle \prod_{m=0}^{p^{n-(\delta+k+1)}-1}\prod_{t_{k}=0}^{p-1}\cdots\prod_{t_{1}=0}^{p-1}\prod_{t_{0}=0}^{p^{\delta}-1}\left(c_{\tau_0},c_{\tau_1},\ldots,c_{\tau_{p-1}}\right),
\end{split}
\end{equation*}
where $\tau_{i}=t_{k+1,mp+i}p^{\delta+k}+((t_{k,mp+i}+t_k)\pmod p)p^{\delta+k-1}+\cdots+((t_{1,mp+i}+t_1)\pmod p)p^{\delta}+((t_{0,mp+i}+t_0)\pmod {p^{\delta}}$, $0 \le t_{k+1,mp+i} \le p^{n-(\delta+k)}-1$ are distinct, $0 \le t_{1,mp+i}, \ldots, t_{k,mp+i} \le p-1$ and $0 \le t_{0,mp+i} \le p^{\delta}-1$.
We next define a map $\sigma_{k+1} \in \mathfrak{S}_q$ as
$$\sigma_{k+1}(\tau_{i})=c_{(mp+i)p^{\delta+k}+t_kp^{\delta+k-1}+\cdots+t_1p^{\delta}+t_0}.$$
One can clearly verify that $\sigma_{k+1}b''_{k+1}\sigma^{-1}_{k+1}=a_{k+1}$ and $\sigma_{k+1}a_{j}\sigma^{-1}_{k+1}=a_{j}$ for $1 \le j \le k$. This completes the proof.
\end{proof}

Consider \(
G = \{a_0^{i_0}a_1^{i_1}\cdots a_{n-\delta}^{i_{n-\delta}} \mid 0 \leq i_0 \leq p^{\delta}-1, 0 \leq i_1,i_2 \dots, i_{n-\delta} \leq p-1\},
\)
where $n\ge 2 $ and  $\delta \in \{1,2\}$, defined as in Theorem \ref{T31}. We denote an $(n-\delta+1)$-tuple $(v_{k,0}, v_{k,1}, \ldots, v_{k,n-\delta}) \in \mathbb{Z}_{p^{\delta}} \times \mathbb{Z}_p \times \mathbb{Z}_p \times \cdots \times \mathbb{Z}_p$ by $V_k = (v_{k,0}, v_{k,1}, \ldots, v_{k,n-\delta})$. For any fixed tuples $V_0, V_1, \ldots, V_{n-\delta} \in \mathbb{Z}_{p^{\delta}} \times \mathbb{Z}_p \times \mathbb{Z}_p \times \cdots \times \mathbb{Z}_p$, define 
\[
\mathcal{N}(V_0, V_1, \ldots, V_{n-\delta}) := \{h \in \mathfrak{S}_q \mid ha_kh^{-1} = a_0^{v_{k,0}} a_1^{v_{k,1}} \cdots a_{n-\delta}^{v_{k,n-\delta}} \text{ for all } 0 \leq k \leq n-\delta\}.
\]

\begin{lem}\label{Atmost_lemma_T31}
Let $G$ be the group defined as in Theorem~\ref{T31}, and $V_0, V_1, \ldots, V_{n-\delta} \in \mathbb{Z}_{p^{\delta}} \times \mathbb{Z}_p \times \mathbb{Z}_p \times \cdots \times \mathbb{Z}_p$, where $\delta \in \{1,2\}$. Then \( |\mathcal{N}(V_0, V_1, \ldots, V_{n-\delta})| \leq p^n \).
\end{lem}

\begin{proof}
Let $ V_0, V_1, \ldots, V_{n-\delta} \in \mathbb{Z}_{p^{\delta}} \times \mathbb{Z}_p \times \mathbb{Z}_p \times \cdots \times \mathbb{Z}_p$ be fixed vectors. We aim to show that \( |\mathcal{N}(V_0, V_1, \ldots, V_{n-\delta})| \leq p^n \). If \( \mathcal{N}(V_0, V_1, \ldots, V_{n-\delta}) = \emptyset \), then the result is trivial. Thus, we assume \( \mathcal{N}(V_0, V_1, \ldots, V_{n-\delta}) \neq \emptyset \).

Define a map \( \zeta : \mathcal{N}(V_0, V_1, \ldots, V_{n-\delta}) \to \mathbb{F}_q \) by \( \zeta(h) := h(c_0) \). It suffices to show that \( \zeta \) is injective. Suppose \( h_1, h_2 \in \mathcal{N}(V_0, V_1, \ldots, V_{n-\delta}) \) and \( \zeta(h_1) = \zeta(h_2) \), i.e., \( h_1(c_0) = h_2(c_0) \). We will show that \( h_1 = h_2 \). Since \( h_1, h_2 \in \mathcal{N}(V_0, V_1, \ldots, V_{n-\delta}) \), we have
\[
h_1 a_k h_1^{-1} = a_0^{v_{k,0}} a_1^{v_{k,1}} \cdots a_{n-\delta}^{v_{k,n-\delta}} \quad \text{and} \quad h_2 a_k h_2^{-1} = a_0^{v_{k,0}} a_1^{v_{k,1}} \cdots a_{n-\delta}^{v_{k,n-\delta}} \quad \text{for all } 0 \leq k \leq n-\delta.
\]
This implies \( h_2^{-1} h_1 a_k h_1^{-1} h_2 = a_k \) for all \( 0 \leq k \leq n-\delta \). Let \( g = h_2^{-1} h_1 \). Then \( g a_k g^{-1} = a_k \) for all \( 0 \leq k \leq n-\delta \) and \( g(c_0) = c_0 \). We aim to show \( g = I \), i.e., $g(c_t)=c_t$, for all $0 \leq t \leq q-1$. We proceed by induction on $t$. For \( t = 0 \), the result is true. Assume \( g(c_t) = c_t \) for \( t < \ell \), where \( 1 \leq \ell \leq p^n - 1 \). We will prove \( g(c_\ell) = c_\ell \). Notice that for every $1 \leq \ell \leq p^n - 1$, either $1 \leq \ell \leq  p^{\delta}-1$ or there exists a unique integer $1 \leq i \leq n-\delta$ such that $ p^{i+\delta-1} \leq \ell \leq p^{i+\delta} - 1 $. First assume that $1 \leq \ell \leq  p^{\delta}-1$. Then $0 \leq \ell -1 < p^{\delta}-1$ and thus Equation \eqref{e31} implies that $a_0(c_{\ell-1})=c_{\ell}$. Since $g(c_{\ell-1})=c_{\ell-1}$ and $ga_0g^{-1}=a_0$, it follows that $ga_0g^{-1}(c_{\ell-1})=a_0(c_{\ell-1})$ and $ga_0g^{-1}(g(c_{\ell-1}))=a_0(c_{\ell-1})$. As $a_{0}(c_{\ell-1})=c_{\ell}$, we conclude that $g(c_{\ell})=c_{\ell}$ for $1 \leq \ell \leq  p^{\delta}-1$. Next, let $\ell$ satisfy $ p^{i+\delta-1} \leq \ell \leq p^{i+\delta} - 1 $, where $1 \leq i \leq n-\delta$. Then, $ 0 \leq \ell-p^{i+\delta-1} \leq p^{i+\delta}-p^{i+\delta-1} - 1=(p-2)p^{i+\delta-1}+p^{i+\delta-1}-1$ and hence there exist integers $j_i$ and $r_i$ such that $\ell-p^{i+\delta-1}=j_i+r_ip^{i+\delta-1}$ with  $0 \leq j_i \leq p^{i+\delta-1}-1$ and $0 \leq r_i \leq p-2$. Therefore, Equation \eqref{e32} implies $a_i(c_{\ell-p^{i+\delta-1}})=a_i(c_{j_i+r_ip^{i+\delta-1}})=a_i(c_{j_i+(r_i+1)p^{i+\delta-1}})=c_{\ell}$. Moreover, using the induction hypothesis, $g(c_{\ell-p^{i+\delta-1}})=c_{\ell-p^{i+\delta-1}}$ and the identity $ga_{i}g^{-1}=a_i$ implies $ga_{i}g^{-1}(g(c_{\ell-p^{i+\delta-1}}))=a_i(c_{\ell-p^{i+\delta-1}})$, i.e., $g(c_{\ell})=c_{\ell}$. It follows that $g = I$, and consequently $h_1 = h_2$. Since $\zeta$ is injective, we conclude that $ |\mathcal{N}(V_0, V_1, \ldots, V_{n-\delta})| \leq p^n$.
\end{proof}

\begin{lem}\label{Atleast_lemma_T31}
Let $\delta \in \{1,2\}$ and $V_0, V_1, \ldots, V_{n-\delta} \in \mathbb{Z}_{p^{\delta}} \times \mathbb{Z}_p \times \mathbb{Z}_p \times \cdots \times \mathbb{Z}_p$. Then $\mathcal{N}(V_0, V_1, \ldots, V_{n-\delta}) \neq \emptyset$ if and only if $p^{\delta-1} \bigm| v_{k,0}$ for all $1 \leq k \leq n-\delta$ and $\det A \neq 0$, where 
\[
A :=
\begin{bmatrix}
R_0\\
R_1\\
\vdots \\
R_{n-\delta}
\end{bmatrix}
=
\begin{bmatrix}
v_{0,0}    \pmod p &  v_{0,1}p^{\delta-1} \pmod p & \cdots & v_{0,n-\delta}p^{\delta-1}\pmod p\\
\frac{v_{1,0}}{p^{\delta-1}} & v_{1,1} & \cdots & v_{1,n-\delta}\\
\vdots & \vdots & \ddots & \vdots\\
\frac{v_{n-\delta,0}}{p^{\delta-1}} & v_{n-\delta,1} & \cdots & v_{n-\delta,n-\delta}
\end{bmatrix}.
\]

Moreover, if $\mathcal{N}(V_0, V_1, \ldots, V_{n-\delta}) \neq \emptyset$, then $|\mathcal{N}(V_0, V_1, \ldots, V_{n-\delta})| = p^n$.
\end{lem}
\begin{proof}
First, we consider the case where $\mathcal{N}(V_0, V_1, \ldots, V_{n-\delta}) \neq \emptyset$, and let $h \in \mathcal{N}(V_0, V_1, \ldots, V_{n-\delta})$. Then we have 
\begin{equation}\label{EE23}
\begin{cases}
ha_0h^{-1} &= a_0^{v_{0,0}} a_1^{v_{0,1}} \cdots a_{n-\delta}^{v_{0,n-\delta}},\\
ha_1h^{-1} &= a_0^{v_{1,0}} a_1^{v_{1,1}} \cdots a_{n-\delta}^{v_{1,n-\delta}},\\  
&\vdots\\
ha_{n-\delta}h^{-1} &= a_0^{v_{n-\delta,0}} a_1^{v_{n-\delta,1}} \cdots a_{n-\delta}^{v_{n-\delta,n-\delta}}.
\end{cases}
\end{equation}

In this case, our aim is to show that $p^{\delta-1} \mid v_{k,0}$ for all $1 \leq k \leq n-\delta$ and that $\det A \neq 0$. The condition $p^{\delta-1} \mid v_{k,0}$ is trivial for all $1 \leq k \leq n-\delta$ if $\delta=1$. For $\delta=2$, assume to the contrary that $p^{\delta-1} \nmid v_{k,0}$ for some $1 \leq k \leq n-\delta$. Then $\gcd(v_{k,0},p^{\delta})<p^{\delta-1}$, and thus $|a_0^{v_{k,0}}|>p$, which implies $\lvert a_0^{v_{k,0}} a_1^{v_{k,1}} \cdots a_{n-\delta}^{v_{k,n-\delta}} \rvert=\lvert ha_{k}h^{-1}\rvert>p $. This is a contradiction because $|a_k| = p$ for all $1 \leq k \leq n-\delta$. Consequently, $p ^{\delta-1}\mid v_{k,0}$ for all $1 \leq k \leq n-\delta$. Next, suppose $p ^{\delta-1}\mid v_{k,0}$ for all $1 \leq k \leq n-\delta$ and $\det A = 0$. This implies that the set $\{R_0, R_1, \ldots, R_{n-\delta}\}$ is linearly dependent over $\mathbb{Z}_p$, i.e., there exist $\alpha_0, \alpha_1, \ldots, \alpha_{n-\delta} \in \mathbb{Z}_p$, not all zero, such that 
\[
\alpha_0R_0 + \alpha_1R_1 + \cdots + \alpha_{n-\delta}R_{n-\delta} = 0.
\]

Without loss of generality, suppose $\alpha_\ell \neq 0$, where $0 \leq \ell \leq n-\delta$. Then, we can write 
\[
R_\ell = \alpha'_0 R_0 + \alpha'_1 R_1 + \cdots + \alpha'_{\ell-1}R_{\ell-1} + \alpha'_{\ell+1}R_{\ell+1} + \cdots + \alpha'_{n-\delta}R_{n-\delta},
\]
where $\alpha'_i = \frac{\alpha_i}{\alpha_\ell}$ in $\mathbb{Z}_p$ for $i \neq \ell$. 

Now, we raise the $0$-th equation of the System~\eqref{EE23} to the power $p^{\delta-1}\alpha'_0$, and the $i$-th equation to the power $\alpha'_i$ for $1 \leq i \leq n-\delta$, where $\alpha'_{\ell}=-1$. Multiplying the resulting equations yields 
\[
a_0^{p^{\delta-1}\alpha'_0} a_1^{\alpha'_1} \cdots a_{n-\delta}^{\alpha'_{n-\delta}}=I,
\]
 which is a contradiction to the fact that $|G|=p^n$. Thus, we conclude that $p^{\delta-1} \mid v_{k,0}$ for $1 \leq k \leq n-\delta$ and $\det A \neq 0$.

 Conversely, we assume that $p^{\delta-1} \mid v_{k,0}$ for all $1 \leq k \leq n-\delta$ and $\det A \neq 0$. For a fixed integer $0 \leq d \leq p^n-1$, we define a function $h_d$ as 
 \begin{equation*}
 \begin{split}
 h_d(c_{t_0+t_1p^{\delta}+t_2p^{\delta+1}+\cdots+t_{n-\delta}p^{n-1}})=&a_{0}^{t_0v_{0,0}+t_{1}v_{1,0}+\cdots+t_{n-\delta}v_{n-\delta,0}}a_{1}^{t_0v_{0,1}+t_{1}v_{1,1}+\cdots+t_{n-\delta}v_{n-\delta,1}}\\&     \cdots a_{n-\delta}^{t_0v_{0,n-\delta}+t_{1}v_{1,n-\delta}+\cdots+t_{n-\delta}v_{n-\delta,n-\delta}}(c_d),
 \end{split}
 \end{equation*}
where $0 \leq t_0 \leq p^{\delta}-1$ and $0 \leq t_k \leq p-1$ for $k \neq 0$. We shall now demonstrate that \( h_d \) is a permutation and \( h_d \in \mathcal{N}(V_0, V_1, \ldots, V_{n-\delta}) \). Suppose that there exist tuples $(t_0, t_1, \ldots, t_{n-\delta}) \neq (t'_0, t'_1, \ldots, t'_{n-\delta})$ such that 
\[
h_d(c_{t_0 + t_1p^{\delta} + t_2p^{\delta+1} + \cdots + t_{n-\delta}p^{n-1}}) = h_d(c_{t'_0 + t'_1p^{\delta} + t'_2p^{\delta+1} + \cdots + t'_{n-\delta}p^{n-1}}).
\]
This leads to the following system of equations:
\[
\begin{cases}
(t_0 - t'_0)v_{0,0} + (t_1 - t'_1)v_{1,0} + \cdots + (t_{n-\delta} - t'_{n-\delta})v_{n-\delta,0} \equiv 0 \pmod{p^{\delta}}, \\
(t_0 - t'_0)v_{0,1} + (t_1 - t'_1)v_{1,1} + \cdots + (t_{n-\delta} - t'_{n-\delta})v_{n-\delta,1} \equiv 0 \pmod{p}, \\
\vdots \\
(t_0 - t'_0)v_{0,n-\delta} + (t_1 - t'_1)v_{1,n-\delta} + \cdots + (t_{n-\delta} - t'_{n-\delta})v_{n-\delta,n-\delta} \equiv 0 \pmod{p}.
\end{cases}
\]

Since $p^{\delta-1} \mid v_{k,0}$ for all $1 \leq k \leq n-\delta$, therefore $p^{\delta-1} \mid (t_0 - t'_0)v_{0,0}$. If $\delta=1$, then clearly $\gcd(v_{0,0},p^{\delta-1})=1$. If $\delta=2$, then first row of matrix $A$ will be $R_0=(v_{0,0}\pmod p,0,\dots,0)$. However, we have $\det A \neq 0$, which implies that $\gcd(v_{0,0},p)=\gcd(v_{0,0},p^{\delta-1})=1$. Thus, $\gcd(v_{0,0},p^{\delta-1})=1$ and $p^{\delta-1} \mid (t_0 - t'_0)$ for $\delta \in \{1,2\}$. This yields that the above system is equivalent to
\[
\begin{cases}
\frac{(t_0 - t'_0)}{p^{\delta-1}}(v_{0,0} \pmod{p}) + (t_1 - t'_1)\frac{v_{1,0}}{p^{\delta-1}} + \cdots + (t_{n-\delta} - t'_{n-\delta})\frac{v_{n-\delta,0}}{p^{\delta-1}} \equiv 0 \pmod{p}, \\
\frac{t_0 - t'_0}{p^{\delta-1}}(v_{0,1}p^{\delta-1} \pmod p)+(t_1 - t'_1)v_{1,1} + \cdots + (t_{n-\delta} - t'_{n-\delta})v_{n-\delta,1} \equiv 0 \pmod{p}, \\
 \vdots \\
\frac{t_0 - t'_0}{p^{\delta-1}}(v_{0,n-\delta}p^{\delta-1} \pmod p)+(t_1 - t'_1)v_{1,n-\delta} + \cdots + (t_{n-\delta} - t'_{n-\delta})v_{n-\delta,n-\delta} \equiv 0 \pmod{p}.
\end{cases}
\]

From this system, we deduce that $A^T V = 0$, where $V = \left( \frac{t_0 - t'_0}{p^{\delta-1}}, t_1 - t'_1, \ldots, t_{n-\delta} - t'_{n-\delta} \right)^T$. Since $\det A \neq 0$, it follows that $V = 0$, which implies $t_k = t'_k$ for $1 \leq k \leq n-\delta$ and $\frac{t_0 - t'_0}{p^{\delta-1}} = 0$ in $\mathbb{Z}_p$. As $\frac{t_0 - t'_0}{p^{\delta-1}} \equiv 0 \pmod{p}$, we have $t_0 - t'_0 \equiv 0 \pmod{p^{\delta}}$, so $t_0 = t'_0$. This contradicts our assumption that \( (t_0, t_1, \ldots, t_{n-\delta}) \neq (t'_0, t'_1, \ldots, t'_{n-\delta}) \). Hence, \( h_d \) is injective, and therefore $h_d$ is permutation.

Next, we show that $h_d \in \mathcal{N}(V_0, V_1, \ldots, V_{n-\delta})$, which is equivalent to show that
\[
h_d a_k(c_{t_0 + t_1p^{\delta}+t_2p^{\delta+1} + \cdots + t_{n-\delta}p^{n-1}}) = a_0^{v_{k,0}} a_1^{v_{k,1}} \cdots a_{n-\delta}^{v_{k,n-\delta}} h_d(c_{t_0 + t_1p^{\delta}+t_2p^{\delta+1} + \cdots + i_{n-\delta}p^{n-1}})
\]
for all $0 \leq k \leq n-\delta$, $t_0 \in \{0,1,\ldots,p^{\delta}-1\}$ and $t_1,t_2 \ldots, t_{n-\delta} \in \{0, 1, \ldots, p-1\}$. We verify this identity only for $k=0$. The remaining cases follow by a similar argument, using Equation~\eqref{e34} and definition of $h_d$.

 Using Equation \eqref{e33}, we obtain
\[
a_0(c_{t_0 + t_1p^{\delta} + t_2p^{\delta+1}+ \cdots + t_{n-\delta}p^{n-1}}) = c_{(t_0+1)\pmod {p^{\delta}} + t_1p^{\delta} + t_2p^{\delta+1}+ \cdots + t_{n-\delta}p^{n-1}}.
\]
Now, using the above equation and definition of $h_{d}$,
\begin{equation*}
\begin{split}
h_da_0(c_{t_0 + t_1p^{\delta} + t_2p^{\delta+1}+ \cdots + t_{n-\delta}p^{n-1}})=&h_d(c_{(t_0+1)\pmod {p^{\delta}} + t_1p^{\delta} + t_2p^{\delta+1}+ \cdots + t_{n-\delta}p^{n-1}})\\=&a_{0}^{(t_0+1)v_{0,0}+t_{1}v_{1,0}+\cdots+t_{n-\delta}v_{n-\delta,0}}\\&a_{1}^{(t_0+1)v_{0,1}+t_{1}v_{1,1}+\cdots+t_{n-\delta}v_{n-\delta,1}}\cdots\\&  a_{n-\delta}^{(t_0+1)v_{0,n-\delta}+t_{1}v_{1,n-\delta}+\cdots+t_{n-\delta}v_{n-\delta,n-\delta}}(c_d),
\\=&a_{0}^{v_{0,0}}a_{0}^{t_0v_{0,0}+t_{1}v_{1,0}+\cdots+t_{n-\delta}v_{n-\delta,0}}\\&a_{1}^{v_{0,1}}a_{1}^{t_0v_{0,1}+t_{1}v_{1,1}+\cdots+t_{n-\delta}v_{n-\delta,1}}\cdots \\&  a_{n-\delta}^{v_{0,n-\delta}}a_{n-\delta}^{t_0v_{0,n-\delta}+t_{1}v_{1,n-\delta}+\cdots+t_{n-\delta}v_{n-\delta,n-\delta}}(c_d)\\=&a_{0}^{v_{0,0}}a_{1}^{v_{0,1}}\cdots a_{n-\delta}^{v_{0,n-\delta}}h_d (c_{t_0 + t_1p^{\delta} + t_2p^{\delta+1}+ \cdots + t_{n-\delta}p^{n-1}})
\end{split}
\end{equation*}
as $|a_0| = p^{\delta}$ and $|a_1|=|a_2|=\cdots=|a_{n-\delta}|=p\mid p^{\delta}$. Thus, $h_d \in \mathcal{N}(V_0, V_1, \ldots, V_{n-\delta})$. Notice that $h_d$ depends on $c_d$, which has $p^n$ choices. Moreover, for $d_1 \neq d_2$, we have $h_{d_1} \neq h_{d_2}$. Hence, using Lemma \ref{Atmost_lemma_T31}, we conclude that $|\mathcal{N}(V_0, V_1, \ldots, V_{n-\delta})| = p^n$.
\end{proof}

\begin{thm}\label{T41}
Let $\delta \in \{1,2\}$, $q = p^n$, where $n\ge 2 $ is a positive integer, and let $f \in \mathbb{F}_q[X, Y]$ be the permutation group polynomial corresponding to the group $G$ in Theorem~\ref{T31}. Furthermore, let $\mathcal{P}_{f}$ denote the number of all permutation group polynomials over $\F_q$ that are of the form $f$. Then 
\[
\displaystyle \mathcal{P}_{f} =
        \dfrac{p^n!(p^n-1)!}{\left(p^{n}-p^{(\delta-1)(n-\delta+1)}\right)\displaystyle\prod_{i=1}^{n-\delta}\left(p^{n-\delta+1}-p^{i}\right)}.
\]
\end{thm}

\begin{proof}
Let $g(X_1, X_2)$ be a permutation group polynomial that is of the form $f$, and let $H$ be its corresponding group. Then from Lemma \ref{conjugate_lemma}, $H = hGh^{-1}$ for some $h \in \mathfrak{S}_q$. Also, every conjugate $H$ of $G$ in $\mathfrak{S}_q$ gives the permutation group polynomials of the form $f$. Furthermore, we know that $p^n!$ permutation group polynomials of the same form can be obtained by permuting the elements of $H$. Therefore, 
\[
\mathcal{P}_{f} = p^n! \times |Conj_{\mathfrak{S}_q}(G)| = \frac{(p^n!)^2}{|\mathfrak{N}_{\mathfrak{S}_q}(G)|},
\]
where $Conj_{\mathfrak{S}_q}(G) = \{hGh^{-1} \mid h \in \mathfrak{S}_q\}$ is the set of all conjugates of $G$ in $\mathfrak{S}_q$, and $\mathfrak{N}_{\mathfrak{S}_q}(G) = \{h \in \mathfrak{S}_q \mid hGh^{-1} = G\}$ is the normalizer of $G$ in $\mathfrak{S}_q$. Now, let us compute the cardinality of $\mathfrak{N}_{\mathfrak{S}_q}(G)$. Since $G = \langle a_0, a_1, \ldots, a_{n-\delta} \rangle$, it follows that 
\[
\mathfrak{N}_{\mathfrak{S}_q}(G) = \bigcup \mathcal{N}(V_0, V_1, \ldots, V_{n-\delta}),
\]
where the union is taken over all the $(n-\delta+1)$-tuples $V_0, V_1, \ldots, V_{n-\delta} \in \mathbb{Z}_{p^{\delta}} \times \mathbb{Z}_p \times \mathbb{Z}_p \times \cdots \times \mathbb{Z}_p$. It is easy to see that for any tuples $(V_0, V_1, \ldots, V_{n-\delta}) \neq (V_0', V_1', \ldots, V_{n-\delta}')$, we have 
\[
\mathcal{N}(V_0, V_1, \ldots, V_{n-\delta}) \cap \mathcal{N}(V_0', V_1', \ldots, V_{n-\delta}') = \emptyset.
\]
Thus, 
\[
|\mathfrak{N}_{\mathfrak{S}_q}(G)| = \sum_{V_0, V_1, \ldots, V_{n-\delta} \in \mathbb{Z}_{p^{\delta}} \times \mathbb{Z}_p \times \mathbb{Z}_p \times \cdots \times \mathbb{Z}_p} |\mathcal{N}(V_0, V_1, \ldots, V_{n-\delta})|.
\]
Using  Lemma \ref{Atleast_lemma_T31}, we have
\[
|\mathfrak{N}_{\mathfrak{S}_q}(G)| = p^n |B|,
\]
where $B:=\{(V_0, V_1, \ldots, V_{n-\delta}) \in \mathbb{Z}_{p^{\delta}} \times \mathbb{Z}_p \times \mathbb{Z}_p \times \cdots \times \mathbb{Z}_p  : p^{\delta-1} \mid v_{k,0} \text{ for all } 1 \leq k \leq n-\delta \text{ and } \det A \neq 0\}$ and $A$ is the matrix defined in Lemma \ref{Atleast_lemma_T31}.

To compute the cardinality of $B$, in other words to count all the tuples \( V_0, V_1, \ldots, V_{n-\delta} \), where \( V_k = (v_{k,0}, v_{k,1}, \ldots, v_{k,n-\delta}) \in \mathbb{Z}_{p^{\delta}} \times \mathbb{Z}_p \times \mathbb{Z}_p \times \cdots \times \mathbb{Z}_p \), \( 0 \leq k \leq n-\delta \), \(p^{\delta-1} \mid v_{k,0}\) for \(k \geq 1\), and \(\det A \neq 0\), i.e., the set \(\{R_0, R_1, \ldots, R_{n-\delta}\}\) is linearly independent over \(\mathbb{Z}_p\), we define a map 
\[
\phi = (\phi_0, \phi_1, \ldots, \phi_{n-\delta}): S_0 \times S_1 \times S_1 \times \cdots \times S_1 \rightarrow S \times S \times S \times \cdots \times S
\]
where \( S_0 = \mathbb{Z}_{p^{\delta}} \times \mathbb{Z}_p \times \cdots \times \mathbb{Z}_p \), \( S_1 = p^{\delta-1}{\mathbb{Z}_{p^{\delta}}} \times \mathbb{Z}_p \times \cdots \times \mathbb{Z}_p \), \( S = \mathbb{Z}_p \times \mathbb{Z}_p \times \cdots \times \mathbb{Z}_p \), \(\phi_0: S_0 \rightarrow S\) is defined as 
\[
\phi_{0}(V_0) := (v_{0,0}  \pmod p, v_{0,1}p^{\delta-1} \pmod p, \ldots, v_{0,n-\delta}p^{\delta-1}\pmod p)=R_0,
\]
and \(\phi_k: S_1 \rightarrow S\) is defined as 
\[
\phi_{k}(V_k) := \left(\frac{v_{k,0}}{p^{\delta-1}}, v_{k,1}, \ldots, v_{k,n-\delta}\right)=R_k
\]
for \(k \geq 1\). 

We notice that $\phi_k$ is a bijection for $k\geq 1$, which renders 
\[
|\phi^{-1}((R_0, R_1, \ldots, R_{n-\delta}))| = \prod_{i=0}^{n-\delta}|\phi_i^{-1}(R_i)| = |\phi_0^{-1}(R_0)|.
\]
Since $\phi_{0}:\mathbb{Z}_{p^{\delta}} \times \mathbb{Z}_p \times \cdots \times \mathbb{Z}_p \rightarrow \mathbb{Z}_p \times \mathbb{Z}_p \times \cdots \times \mathbb{Z}_p$ such that
\[
\phi_0(v_{0,0}, v_{0,1}, \ldots, v_{0,n-\delta}) = (v_{0,0}  \pmod p, v_{0,1}p^{\delta-1} \pmod p, \ldots, v_{0,n-\delta}p^{\delta-1}\pmod p),
\] thus $\phi_{0}$ can be written as $\phi_{0}=(\phi_{00},\phi_{01},\ldots,\phi_{0n-\delta})$, where $\phi_{00}:\mathbb{Z}_{p^{\delta}}\rightarrow \mathbb{Z}_p$, $\phi_{00}(v_{0,0})=v_{0,0}  \pmod p$ and for each $1 \leq k \leq n-\delta$, $\phi_{0k}:\mathbb{Z}_p \rightarrow \mathbb{Z}_p$, which is defined as $\phi_{0k}(v_{0,k})=v_{0,k}p^{\delta-1} \pmod p$. All $\phi_{00}$ and $\phi_{0k}$, $1\leq k \leq n-\delta$, are group homomorphisms such that kernel of each homomorphism has exactly $p^{\delta-1}$ elements, thus all these maps are $p^{\delta-1}$ to $1$. Consequently,
\[
|\phi^{-1}((R_0, R_1, \ldots, R_{n-\delta}))| = |\phi_0^{-1}(R_0)|=\prod_{k=0}^{n-\delta}|\phi_{0k}^{-1}(v_{0,k})|=p^{\delta-1}p^{\delta-1}\cdots p^{\delta-1}=p^{(\delta-1)(n-\delta+1)}.
\]
Next, we count all matrices $A$ such that $\det A \neq 0$ (or equivalently, \(\{R_0, R_1, \ldots, R_{n-\delta}\}\) is linearly independent over \(\mathbb{Z}_p\)). For $\delta \in \{1,2\}$, $p^{\delta-1}\mid p$ and there are $\frac{p}{p^{\delta-1}}$ multiples of $p^{\delta-1}$ in $\mathbb{Z}_p$, which yields that there are exactly $p\times \underbrace{\frac{p}{p^{\delta-1}}\times \frac{p}{p^{\delta-1}}\times  \cdots \times \frac{p}{p^{\delta-1}}}_{(n-\delta)-\text{ times }}-1 =\dfrac{p^{n-\delta+1}-p^{(\delta-1)(n-\delta)}}{p^{(\delta-1)(n-\delta)}}$ choices for $R_{0}$ as $\det A\neq 0$ and $v_{0,0}\pmod p$ has $p$ choices. Furthermore, each of $\frac{v_{k,0}}{p^{\delta-1}}$ and $v_{k,i}$, where $1 \leq i,k \leq n-\delta$, also has $p$ choices. Thus, for given $R_0$ we have $p^{n-\delta+1}-p$ choices for $R_1$ such that $\det A\neq 0$. Continuing in this way we have 
\[
\dfrac{p^{n-\delta+1}-p^{(\delta-1)(n-\delta)}}{p^{(\delta-1)(n-\delta)}}\times (p^{n-\delta+1}-p) \times (p^{n-\delta+1}-p^{2})\times \cdots \times  (p^{n-\delta+1}-p^{n-\delta})  
\]
choices for $R_0, R_1, \ldots, R_{n-\delta}$ such that $\det A \neq 0$.
Therefore from the above discussion, 
\[
|B|=p^{(\delta-1)(n-\delta+1)}\dfrac{p^{n-\delta+1}-p^{(\delta-1)(n-\delta)}}{p^{(\delta-1)(n-\delta)}}\prod_{i=1}^{n-\delta}(p^{n-\delta+1}-p^{i})=(p^{n}-p^{(\delta-1)(n-\delta+1)})\prod_{i=1}^{n-\delta}(p^{n-\delta+1}-p^{i}).
\]
Thus, 
\[
|\mathfrak{N}_{\mathfrak{S}_q}(G)|=p^{n}(p^{n}-p^{(\delta-1)(n-\delta+1)})\prod_{i=1}^{n-\delta}(p^{n-\delta+1}-p^{i})
\]
and finally, we obtain
\[
\mathcal{P}_{f} =
       \dfrac{p^n!(p^n-1)!}{\left(p^{n}-p^{(\delta-1)(n-\delta+1)}\right)\displaystyle \prod_{i=1}^{n-\delta}\left(p^{n-\delta+1}-p^{i}\right)},
\]
and the proof is complete.
\end{proof}
The following theorem counts the permutation group polynomials equivalent to a given permutation group polynomial or number of permutation polynomial tuples equivalent to its corresponding permutation polynomial tuple.
\begin{thm}\label{Equivalent_PGP}
Let $f$ be a permutation group polynomial and its corresponding permutation polynomial tuple is $\underline{\beta}_f=(\beta_0,\beta_1,\ldots,\beta_{q-1})\in \mathfrak{S}_q^{q}$. Then there are $\frac{q(q!)}{|\mathfrak{C}_{\mathfrak{S}_q}(G)|}$ permutation group polynomials that are equivalent to $f$, where $G=\{\beta_0,\beta_1,\ldots,\beta_{q-1}\}$ and $\mathfrak{C}_{\mathfrak{S}_q}(G)$ is the centralizer of $G$ in $\mathfrak{S}_q$.
\end{thm}
\begin{proof}
Let $g(X_1,X_2)\in \F_q[X_1,X_2]$ be a permutation group polynomial, which is equivalent to $f$ and $\underline{\gamma}_g=(\gamma_0,\gamma_1,\ldots,\gamma_{q-1})$ is the permutation polynomial tuple of $g$. Then there exist $\sigma, \lambda \in \mathfrak{S}_q$ such that $\sigma (\beta_0,\beta_1,\ldots,\beta_{q-1})\lambda=(\gamma_0,\gamma_1,\ldots,\gamma_{q-1})$. Since one of permutations $\beta_i$ is the identity, there exists $0 \leq i \leq q-1$ such that $\sigma \lambda=\gamma_i$. Consequently, $$\sigma (\beta_0,\beta_1,\ldots,\beta_{q-1})\lambda=\sigma (\beta_0,\beta_1,\ldots,\beta_{q-1})\sigma^{-1}\gamma_i=(\gamma_0,\gamma_1,\ldots,\gamma_{q-1}).$$ As $g(X_1,X_2)$ is a permutation group polynomial, the set $\{ \gamma_0,\gamma_1,\ldots,\gamma_{q-1}\}$ forms a subgroup of $\mathfrak{S}_q$. Thus, $\sigma (\beta_0,\beta_1,\ldots,\beta_{q-1})\sigma^{-1}\gamma_i=(\gamma_0,\gamma_1,\ldots,\gamma_{q-1})$ implies that the right coset $\sigma \{ \beta_0,\beta_1,\ldots,\beta_{q-1}\}\sigma^{-1}\gamma_i$  of the subgroup $\sigma \{ \beta_0,\beta_1,\ldots,\beta_{q-1}\}\sigma^{-1}$ corresponding to $\gamma_i$ forms a subgroup of $\mathfrak{S}_q$, which is true if and only if $\gamma_i \in \{\sigma \beta_0\sigma^{-1},\sigma \beta_1\sigma^{-1},\ldots,\sigma \beta_{q-1}\sigma^{-1}\}$. Therefore, the number of permutation group polynomials that are equivalent to $f$ is same as the 
cardinality of the set 
$$A:=\{\sigma (\beta_0,\beta_1,\ldots,\beta_{q-1})\sigma^{-1} \sigma \beta_i \sigma^{-1} \mid \sigma \in \mathfrak{S}_q, 0 \leq i\leq q-1\}.$$
It is easy to see that $A=\displaystyle \bigcup_{k=1}^{q!}A_{\sigma_k}$, where $\{\sigma_1, \sigma_2, \dots, \sigma_{q!}\}=\mathfrak S_{q}$ and for each $\sigma_k \in \mathfrak S_q$
$$A_{\sigma_k}:=\{\sigma_k (\beta_0,\beta_1,\ldots,\beta_{q-1}){\sigma_k}^{-1} \sigma_k \beta_i {\sigma_k}^{-1} \mid 0 \leq i \leq q-1\}.$$
Since one of the permutations $\beta_i$ is the identity, it follows immediately that $|A_{\sigma_k}|=q$ for $1\leq k \leq q!$. We claim that for any $1\leq j \neq j' \leq q!$, either $A_{\sigma_j} \cap A_{\sigma_{j'}}=\emptyset$ or $A_{\sigma_j}= A_{\sigma_{j'}}$. If $A_{\sigma_j} \cap A_{\sigma_{j'}}=\emptyset$, then the claim holds trivially. Now, let $A_{\sigma_j} \cap A_{\sigma_{j'}}\neq \emptyset$, which implies that there exist some $\beta_i$ and $\beta_{i'}$, $0\leq i, i'\leq q-1$, such that 
\begin{equation} \label{equi}
\sigma_j(\beta_0,\beta_1,\ldots,\beta_{q-1})\sigma_j^{-1} \sigma_j\beta_i \sigma_j^{-1}=\sigma_{j'} (\beta_0,\beta_1,\ldots,\beta_{q-1}){\sigma_{j'}}^{-1} \sigma_{j'}\beta_{i'} {\sigma_{j'}}^{-1}.
\end{equation}
As one of the permutations $\beta_i$ is the identity, Equation~\eqref{equi} gives that  $\sigma_j\beta_i \sigma_j^{-1}=\sigma_{j'}\beta_{i'} {\sigma_{j'}}^{-1}$. Thus from Equation \eqref{equi}, we have 
$$\sigma_j(\beta_0,\beta_1,\ldots,\beta_{q-1})\sigma_j^{-1} =\sigma_{j'} (\beta_0,\beta_1,\ldots,\beta_{q-1}){\sigma_{j'}}^{-1} $$
which gives that
$A_{\sigma_j}=A_{\sigma_{j'}}$. Moreover, $A_{\sigma_j}=A_{\sigma_{j'}}$ if and only if $\sigma_{j'}^{-1} \sigma_j \in \mathfrak{C}_{\mathfrak{S}_q}(G)$. Consequently,
 $$A=\coprod A_{\sigma_{k'}},$$
 where the disjoint union $\coprod$ runs over a set of the distinct cosets of $\mathfrak{C}_{\mathfrak{S}_q}(G)$ in $ \mathfrak {S}_q$. Thus,
 $$|A|=\sum |A_{\sigma_{k'}}|=\frac{|\mathfrak{S}_q|}{|\mathfrak{C}_{\mathfrak{S}_q}(G)|}q.$$
 Hence, there are $\frac{q(q!)}{|\mathfrak{C}_{\mathfrak{S}_q}(G)|}$ permutation group polynomials that are equivalent to $f$.
\end{proof}

\begin{cor}
Let $f$ be a permutation group polynomial constructed in Theorem \ref{T31} and its corresponding permutation polynomial tuple is $\underline{\beta}_f=(\beta_0,\beta_1,\ldots,\beta_{q-1})\in \mathfrak{S}_q^{q}$. Then there are $q!$ permutation group polynomials equivalent to $f$.
\end{cor}
\begin{proof}
It is easy to verify that for $G$ defined in Theorem \ref{T31}, $\mathfrak{C}_{\mathfrak{S}_q}(G)=\mathcal{N}(V_0, V_1, \ldots, V_{n-\delta})$, where $V_i=(0,\ldots,0,1,0,\ldots,0)$ denotes the vector with a single nonzero entry in the
$i$-th position. Moreover, from Lemma \ref{Atleast_lemma_T31}, we have $|\mathcal{N}(V_0, V_1, \ldots, V_{n-\delta})|=p^n=q$, thus $|\mathfrak{C}_{\mathfrak{S}_q}(G)|=q$, which completes the proof. 
  \end{proof}

Next, we enumerate all $e$-Klenian polynomials over $\mathbb{F}_q$, for $e \geq 1$,  with the following lemmas serving as key tools in this process.
\begin{lem}\label{L41}
Let $q=p^n$,  where $n \geq 2$. Let \( g(X_1,X_2) \in \F_q[X_1,X_2] \) be a permutation group polynomial and its corresponding group $H < \mathfrak{S}_q$ be isomorphic to $G$, where $G$ is the subgroup of $\mathfrak{S}_q$ defined in Lemma \ref{JJ_tpl}. Then there exists $\mu \in \mathfrak{S}_q$ such that $\mu H \mu^{-1}=G$.
\end{lem}

\begin{proof}
Let \( f_{e}(X_1,X_2) \) be the \(e\)-Klenian polynomial and \( G \) be the corresponding group of \( f_e \) over \(\mathbb{F}_{q}\). Then from Lemma \ref{JJ_tpl},  \( G = \langle a, b \rangle \), where 
\[
a = (c_0, c_1, \ldots, c_{p^e-1})(c_{p^e}, c_{p^e+1}, \ldots, c_{2p^e-1}) \cdots (c_{p^n-p^e}, c_{p^n-p^e+1}, \ldots, c_{p^n-1}),
\]
and
\[
b = (c_0, c_{p^e}, \ldots, c_{p^n-p^e})(c_1, c_{p^e+1}, \ldots, c_{p^n-p^e+1}) \cdots (c_{p^e-1}, c_{2p^e-1}, \ldots, c_{p^n-1}).
\]

It is clear that \( H \cong \mathbb{Z}_{p^e} \times \mathbb{Z}_{p^{n-e}} \) as $H\cong G$ and $G \cong \mathbb{Z}_{p^e} \times \mathbb{Z}_{p^{n-e}}$. Thus \( H = \langle a_1, b_1 \rangle \), where \( a_1 \) is a product of \( p^{n-e} \) disjoint cycles, each of length \( p^e \), and \( b_1 \) is a product of \( p^e \) disjoint cycles, each of length \( p^{n-e} \), and they commute with each other. 

Now, we shall show that \( G = \mu H \mu^{-1} \) for some \( \mu \in \mathfrak{S}_{q} \). Since \( a \) and \( a_1 \) have the same cyclic structure, there must exist \( \sigma \in \mathfrak{S}_{q} \) such that \( a = \sigma a_1 \sigma^{-1} \). This implies that \(  \sigma H \sigma^{-1} = \langle a, d \rangle \), where \( d = \sigma b_1 \sigma^{-1} \). As \( H \) is the group corresponding to a permutation group polynomial, so is \( \sigma H \sigma^{-1} \). Therefore, \( \sigma H \sigma^{-1} \) must satisfy the following properties
\begin{enumerate}
    \item \( a \) and \( d \) commute, and \( d \notin \langle a \rangle \),
    \item \( d \) is a product of \( p^e \) disjoint cycles, each of length \( p^{n-e} \), and
    \item no element, except the identity, has a fixed point.
\end{enumerate}
Using these properties and similar argument as in Lemma \ref{conjugate_lemma}, the cyclic structure of \( d \) can be expressed as
\begin{equation*}
\begin{split}
d = &\ (c_{0}, c_{i_1p^e + j_1}, c_{i_2p^e + j_2}, \ldots, c_{i_{p^{n-e}-1}p^e + j_{p^{n-e}-1}}) \\
&\ (c_{1}, c_{i_1p^e + (j_1 + 1) \pmod {p^e}}, c_{i_2p^e + (j_2 + 1) \pmod {p^e}}, \ldots, c_{i_{p^{n-e}-1}p^e + (j_{p^{n-e}-1} + 1) \pmod {p^e}}) \\
&\ \cdots \\
&\ (c_{p^e-1}, c_{i_1p^e + (j_1 + p^e - 1) \pmod {p^e}}, c_{i_2p^e + (j_2 + p^e - 1) \pmod {p^e}}, \ldots, c_{i_{p^{n-e}-1}p^e + (j_{p^{n-e}-1} + p^e - 1) \pmod {p^e}}),
\end{split}
\end{equation*}
where \( \{i_1, i_2, \ldots, i_{p^{n-e}-1}\} = \{1, 2, \ldots, p^{n-e}-1\} \) and \( j_1, j_2, \ldots, j_{p^{n-e}-1} \in \{0, 1, \ldots, p^e-1\} \).

Now, for a fixed \( k \in \{0, 1, \ldots, p^e-1\} \), we define a map \( \phi \in \mathfrak{S}_q \) as follows
\[
\phi(c_{i_tp^e + (j_t + k) \pmod {p^e}}) = c_{tp^e +  k}
\]
for \( t \in \{1, 2, \ldots, p^{n-e}-1\} \), and \( \phi(c_j) = c_j \) for all \( j \in \{0, 1, \ldots, p^e-1\} \).
One can easily show that $\phi^{-1} a \phi= a$ and $\phi d \phi^{-1} = b$. Therefore, $\phi$ satisfies \( \phi a \phi^{-1} = a \),  \( \phi d \phi^{-1} = b \) and thus $G=\mu H \mu^{-1}$, where $\mu=\phi\sigma$.
\end{proof}
\begin{rmk}\label{R41}
From the above lemma, we can conclude that any permutation group polynomial over $\F_{p^2}$ is an $e$-Klenian polynomial.
\end{rmk}

From Lemma \ref{L41}, we observe that $e$-Klenian polynomial and $(n-e)$-Klenian polynomial over $\F_{q}$ are of the same form. Therefore, without loss of generality, we assume \( n-e \geq e \) for $e$-Klenian polynomials over $\F_{q}$.

Next, let \( G = \langle a, b \rangle \) be the subgroup of $\mathfrak{S}_q$ defined as in Lemma \ref{JJ_tpl}. For fixed \( i, u \in \{0, 1, \ldots, p^e-1\} \) and \( j, v \in \{0, 1, \ldots, p^{n-e}-1\} \), where $e\geq 1$, we define the set
\[
\mathcal{N}(i, j, u, v) := \{ h \in \mathfrak{S}_q \mid h a h^{-1} = a^i b^j, \, h b h^{-1} = a^u b^v \}.
\]

\begin{lem}\label{Atmost_lemma}
Let \( q = p^n \) and \( G = \langle a, b \rangle \) be a subgroup of \( \mathfrak{S}_q \) as defined in Lemma \ref{JJ_tpl}. Moreover, let \( i, u \in \{0, 1, \ldots, p^e-1\} \) and \( j, v \in \{0, 1, \ldots, p^{n-e}-1\} \) be fixed integers. Then \( |\mathcal{N}(i, j, u, v)| \leq p^n \).
\end{lem}

\begin{proof}
Let \( i, u \in \{0,1,\ldots,p^e-1\} \), and \( j, v \in \{0,1,\ldots,p^{n-e}-1\} \) be fixed integers. If \( \mathcal{N}(i,j,u,v) = \emptyset \), then there is nothing to prove. Hence, we assume that \( \mathcal{N}(i,j,u,v) \neq \emptyset \).  We define a map \( \psi: \mathcal{N}(i,j,u,v) \rightarrow \mathbb{F}_{q} \) as  
\[
  \psi(g) = g(c_0).
\]
It is sufficient to show that \( \psi \) is injective. Suppose that there exist \( h_1, h_2 \in \mathcal{N}(i,j,u,v) \) such that \( \psi(h_1) = \psi(h_2) \), i.e., \( h_1(c_0) = h_2(c_0) \).  Since \( h_1, h_2 \in \mathcal{N}(i,j,u,v) \), the structure of \( a \) and \( b \) implies the following decomposition into cycles  
\begin{align*}
(h_1(c_0), h_1(c_{p^e}), \ldots, h_1(c_{p^n-p^e}))(h_1(c_1), h_1(c_{p^e+1}), \ldots, h_1(c_{p^n-p^e+1})) \cdots &\\ (h_1(c_{p^e-1}), h_1(c_{2p^e-1}), \ldots, h_1(c_{p^n-1})) &= a^u b^v, \\
(h_1(c_0), h_1(c_1), \ldots, h_1(c_{p^e-1}))(h_1(c_{p^e}), h_1(c_{p^e+1}), \ldots, h_1(c_{2p^e-1})) \cdots & \\ (h_1(c_{p^n-p^e}), h_1(c_{p^n-p^e+1}), \ldots, h_1(c_{p^n-1})) &= a^i b^j.
\end{align*}

Similarly, for \( h_2 \), we have
\begin{align*}
(h_2(c_0), h_2(c_{p^e}), \ldots, h_2(c_{p^n-p^e}))(h_2(c_1), h_2(c_{p^e+1}), \ldots, h_2(c_{p^n-p^e+1})) \cdots &\\(h_2(c_{p^e-1}), h_2(c_{2p^e-1}), \ldots, h_2(c_{p^n-1})) & = a^u b^v,\\
(h_2(c_0), h_2(c_1), \ldots, h_2(c_{p^e-1}))(h_2(c_{p^e}), h_2(c_{p^e+1}), \ldots, h_2(c_{2p^e-1})) \cdots & \\(h_2(c_{p^n-p^e}), h_2(c_{p^n-p^e+1}), \ldots, h_2(c_{p^n-1})) & = a^i b^j.
\end{align*}

From the above cyclic structures and the fact that \( h_1(c_0) = h_2(c_0) \), we observe the following 
\[
h_1(c_s) = (a^i b^j)^s(h_1(c_0)) = (a^i b^j)^s(h_2(c_0)) = h_2(c_s), \quad \text{for } 0 \leq s \leq p^e-1.
\]
Moreover,  
\[
h_1(c_{rp^e+s}) = (a^u b^v)^r(h_1(c_s)) \text{ and } h_2(c_{rp^e+s}) = (a^u b^v)^r(h_2(c_s)),
\]
for \( 0 \leq s \leq p^e-1 \) and \( 1 \leq r \leq p^{n-e}-1 \). Since \( h_1(c_s) = h_2(c_s) \) for \( 0 \leq s \leq p^e-1 \), we have
\[
h_1(c_{rp^e+s})=(a^ub^v)^{r}(h_1(c_s))=(a^ub^v)^{r}(h_2(c_s))=h_2(c_{rp^e+s}),
\]
for \( 0 \leq s \leq p^e-1 \) and \( 1 \leq r \leq p^{n-e}-1 \).  Thus, \( h_1(c_t) = h_2(c_t) \) for all \( 0 \leq t \leq p^n-1 \), which implies \( h_1 = h_2 \).  Therefore, the map \( \psi \) is injective, and we conclude that \( |\mathcal{N}(i,j,u,v)| \leq p^n \).
\end{proof}
\begin{lem}\label{Atleast_lemma}
Let \( q = p^n \) and \( G = \langle a, b \rangle \) be a subgroup of \( \mathfrak{S}_q \) as defined in Lemma \ref{JJ_tpl}. Furthermore, let \( i, u \in \{0, 1, \ldots, p^e-1\} \) and \( j, v \in \{0, 1, \ldots, p^{n-e}-1\} \) such that \( p^{n-2e} \mid j \). Then, \( \mathcal{N}(i, j, u, v) \neq \emptyset \) if and only if 
\[
ip^{\alpha} \not\equiv \beta u \pmod{p^e} \quad \text{or} \quad jp^{\alpha} \not\equiv \beta v \pmod{p^{n-e}}
\]
for all \( (\beta, \alpha) \in \mathbb{Z} \times \{0, 1, \dots, e-1\} \).  In this case, \( |\mathcal{N}(i, j, u, v)| = p^n \).
\end{lem}
\begin{proof}
Let \( i, u \in \{0, 1, \ldots, p^e-1\} \) and \( j, v \in \{0, 1, \ldots, p^{n-e}-1\} \) be integers such that \( p^{n-2e} \mid j \). Our aim is to show that \( \mathcal{N}(i, j, u, v) \neq \emptyset \) if and only if  
\[
\begin{aligned}
& ip^{\alpha} \not\equiv \beta u \pmod{p^e}, \text{ or} \\
& jp^{\alpha} \not\equiv \beta v \pmod{p^{n-e}}
\end{aligned}
\]
for all \( (\beta, \alpha) \in \mathbb{Z} \times \{0, 1, \ldots, e-1\} \). 

First, we assume that \( \mathcal{N}(i, j, u, v) \neq \emptyset \) and let \( h \in \mathcal{N}(i, j, u, v) \). If \( ip^{\alpha} \not\equiv \beta u \pmod{p^e} \) for any \( \beta \in \mathbb{Z} \), \( 0 \leq \alpha \leq e-1 \), then there is nothing to prove. Therefore, we can assume that there exists some \( \beta_1 \in \mathbb{Z} \) and \( 0 \leq \alpha_1 \leq e-1 \) such that \( ip^{\alpha_1} \equiv \beta_1 u \pmod{p^e} \). Further, if \( jp^{\alpha_1} \not\equiv \beta_1 v \pmod{p^{n-e}} \), then we are done. Therefore, we suppose that \( ip^{\alpha_1} \equiv \beta_1 u \pmod{p^e} \) and \( jp^{\alpha_1} \equiv \beta_1 v \pmod{p^{n-e}} \). Since \( h \in \mathcal{N}(i, j, u, v) \), we have \( ha^{p^{\alpha_1}}h^{-1} = hb^{\beta_1}h^{-1} \), i.e., \( a^{p^{\alpha_1}} = b^{\beta_1} \), which is not possible as \( 0 \leq \alpha_1 \leq e-1 \).

Conversely, we suppose that \( ip^{\alpha} \not\equiv \beta u \pmod{p^e} \) or \( jp^{\alpha} \not\equiv \beta v \pmod{p^{n-e}} \) for all \( (\beta, \alpha) \in \mathbb{Z} \times \{0,1,\ldots,e-1\} \). Now, we shall show that \( \mathcal{N}(i,j,u,v) \neq \emptyset \). For a fixed integer \( d \in \{0,1,\ldots,q-1\} \), we define a function
\[
h_d(c_{rp^e+s}) = a^{ur+is}b^{vr+js}(c_d),
\]
where \( r \in \{0,1,\ldots, p^{n-e}-1\} \) and \( s \in \{0,1,\ldots, p^e-1\} \). Next, we will show that \( h_d \) is a permutation and \( h_d \in \mathcal{N}(i,j,u,v) \). First, on the contrary, assume that there exist \( r_1, r_2 \in \{0,1,\ldots, p^{n-e}-1\} \) and \( s_1, s_2 \in \{0,1,\ldots,p^e-1\} \) such that \( (r_1,s_1) \neq (r_2,s_2) \) and 
\[
h_d(c_{r_1p^e+s_1}) = h_d(c_{r_2p^e+s_2}),
\]
which gives
\[
a^{u(r_1-r_2)+i(s_1-s_2)}b^{v(r_1-r_2)+j(s_1-s_2)}(c_d) = c_d.
\]
This implies
\[
\begin{aligned}
& u(r_1-r_2) + i(s_1-s_2) \equiv 0 \pmod{p^e}, \\
& v(r_1-r_2) + j(s_1-s_2) \equiv 0 \pmod{p^{n-e}},
\end{aligned}
\]
where \( |r_1-r_2| \leq p^{n-e}-1 \) and \( |s_1-s_2| \leq p^e-1 \). This is equivalent to
\[
\begin{aligned}
& ip^{\alpha} \equiv \frac{r_2-r_1}{x} u \pmod{p^e}, \\
& jp^{\alpha} \equiv \frac{r_2-r_1}{x} v \pmod{p^{n-e}},
\end{aligned}
\]
where \( s_1-s_2 = p^{\alpha}x \) for some \( 0 \leq \alpha \leq e-1 \) such that \( p \nmid x \). Thus, for \( \beta \equiv \frac{r_2-r_1}{x} \pmod{p^{n-e}} \), we get \( ip^{\alpha} \equiv \beta u \pmod{p^e} \) and \( jp^{\alpha} \equiv \beta v \pmod{p^{n-e}} \), which is a contradiction. Therefore, \( h_d \) is a permutation.

Next, we will show that \( h_d \in \mathcal{N}(i,j,u,v) \). It is easy to see that \( h_d \in \mathcal{N}(i,j,u,v) \) if and only if \( h_d a = a^{i}b^{j}h_d \) and \( h_d b = a^{u}b^{v}h_d \). From Lemma \ref{JJ_tpl}, for any \( s \in \{0,1,\ldots,p^e-1\} \) and \( r \in \{0,1,\ldots,p^{n-e}-1\} \), we have
\[
a(c_{rp^e+s}) =
\begin{cases}
c_{rp^e+s+1} & \text{if } s < p^e-1, \\
c_{rp^e} & \text{if } s = p^e-1,
\end{cases}
\quad \text{ and } \quad
b(c_{rp^e+s}) =
\begin{cases}
c_{(r+1)p^e+s} & \text{if } r < p^{n-e}-1, \\
c_s & \text{if } r = p^{n-e}-1.
\end{cases}
\]

Now,
\[
h_d a(c_{rp^e+s}) =
\begin{cases}
h_d(c_{rp^e+s+1}) & \text{if } s < p^e-1, \\
h_d(c_{rp^e}) & \text{if } s = p^e-1,
\end{cases}
=
\begin{cases}
a^{ur+i(s+1)}b^{vr+j(s+1)}(c_d) & \text{if } s < p^e-1, \\
a^{ur}b^{vr}(c_d) & \text{if } s = p^e-1.
\end{cases}
\]

On the other hand,
\[
a^{i}b^{j}h_d(c_{rp^e+s}) = a^{ur+i(s+1)}b^{vr+j(s+1)}(c_d).
\]
If \( s < p^e-1 \), then clearly \( h_d a(c_{rp^e+s}) = a^{i}b^{j}h_d(c_{rp^e+s}) \). For \( s = p^e-1 \), we have
\[
a^{i}b^{j}h_d(c_{rp^e+s}) = a^{ur+ip^e}b^{vr+jp^e}(c_d) = a^{ur}b^{vr}(c_d),
\]
as \( p^{n-2e} \mid j \) and \( |a| = p^e \), \( |b| = p^{n-e} \). Therefore, \( h_d a(c_{rp^e+s}) = a^{i}b^{j}h_d(c_{rp^e+s}) \) for all \( s \in \{0,1,\ldots,p^e-1\} \) and \( r \in \{0,1,\ldots,p^{n-e}-1\} \). Similarly, we can show that \( h_d b(c_{rp^e+s}) = a^{u}b^{v}h_d(c_{rp^e+s}) \), using \( e \leq n-e \) and \( |a| = p^e \), \( |b| = p^{n-e} \). Therefore, \( h_d \in \mathcal{N}(i,j,u,v) \). In this proof, we have seen that \( h_d \) depends on \( c_d \), where \( d \in \{0,1,\ldots,p^n-1\} \), and it is easy to see that for any \( d_1 \neq d_2 \in \{0,1,\ldots,p^n-1\} \), \( h_{d_1} \neq h_{d_2} \). Thus, using Lemma \ref{Atmost_lemma}, we have \( |\mathcal{N}(i,j,u,v)| = p^n \).
\end{proof}

The following theorem determines the exact number of $e$-Klenian polynomials for $e \geq 1$.
\begin{thm}\label{T42}
Let $q = p^n$, where $n \geq 2$ is a positive integer, and let $1 \leq e \leq \lfloor \frac{n}{2} \rfloor$ be an integer. Furthermore, let $\mathcal{P}_{f_e}$ denote the number of $e$-Klenian polynomials  $f_e$ over $\mathbb{F}_q$. Then   
\[
\mathcal{P}_{f_e} =
\begin{cases}
         \dfrac{(p^n)!(p^n-1)!}{ p^{2e}\varphi(p^e)\varphi(p^{n-e})} & \text{if } e < \frac{n}{2},\\
       
        &\\
        \dfrac{(p^n)!(p^n-1)!}{ \varphi(p^e)^2(p^{2e}+p^{2e-1})} & \text{if } e = \frac{n}{2} \text{ and } n \text{ is even},
\end{cases}
\]
where $\varphi(\cdot)$ represents Euler's totient function.
\end{thm}
\begin{proof}
Using a similar argument as in Theorem \ref{T41}, we have
\[
\mathcal{P}_{f_e} = p^n! \times |Conj_{\mathfrak{S}_q}(G)| = \dfrac{(p^n!)^2}{|\mathfrak{N}_{\mathfrak{S}_q}(G)|},
\]
where $Conj_{\mathfrak{S}_q}(G) = \{hGh^{-1} \mid h \in \mathfrak{S}_q\}$ is the set of all conjugates of $G$ in $\mathfrak{S}_q$, and $\mathfrak{N}_{\mathfrak{S}_q}(G) = \{h \in \mathfrak{S}_q \mid hGh^{-1} = G\}$ is the normalizer of $G$ in $\mathfrak{S}_q$.

Now, we compute the cardinality of $\mathfrak{N}_{\mathfrak{S}_q}(G)$. Since $G = \langle a, b \rangle$, we have
\[
\mathfrak{N}_{\mathfrak{S}_q}(G) = \{g \in \mathfrak{S}_q \mid gag^{-1} \in G \text{ and } gbg^{-1} \in G \} = \bigcup \mathcal{N}(i,j,u,v),
\]
where the union runs over all $i, u \in \{0, 1, \ldots, p^e-1\}$ and $j, v \in \{0, 1, \ldots, p^{n-e}-1\}$. Again, by similar reasoning as in Theorem \ref{T41}, we have
\[
|\mathfrak{N}_{\mathfrak{S}_q}(G)| = \sum |\mathcal{N}(i,j,u,v)|,
\]
where the summation runs over all $i, u \in \{0, 1, \ldots, p^e-1\}$ and $j, v \in \{0, 1, \ldots, p^{n-e}-1\}$.

Now, we consider two cases: $e < n-e$ and $n-e = e$.
  
\textbf{Case 1:} In this case, we assume that $e < n-e$. It is easy to see that for any $i, u \in \{0, 1, \ldots, p^e-1\}$ and $j, v \in \{0, 1, \ldots, p^{n-e}-1\}$, if $\mathcal{N}(i,j,u,v) \neq \emptyset$, then $|a^{u}b^{v}| = p^{n-e}$ and $|a^{i}b^{j}| = p^e$. This gives $(i, j, u, v) \in U = U_1 \coprod U_2$, where
\begin{equation*}
\begin{split}
U_1 &= \{(i, j, u, v) \mid \gcd(p^{n-e}, v) = \gcd(p, v) = 1, j = k_1p^{n-2e}, \gcd(k_1, p) = 1, 0 \leq k_1 \leq p^e-1 \}, \\ 
U_2 &= \{(i, j, u, v) \mid \gcd(p^{n-e}, v) = \gcd(p, v) = 1, \gcd(i, p^e) = 1, j = k_2p^{n-2e+1}, 0 \leq k_2 \leq p^{e-1}-1 \}.
\end{split}
\end{equation*}
Therefore,
\[
|\mathfrak{N}_{\mathfrak{S}_q}(G)| = \sum |\mathcal{N}(i,j,u,v)| = \sum_{(i,j,u,v) \in U} |\mathcal{N}(i,j,u,v)|.
\] 

We can see that if $(i, j, u, v) \in U$, then $p^{n-2e} \mid j$. Therefore, using Lemma \ref{Atleast_lemma}, for any $(i, j, u, v) \in U$, $\mathcal{N}(i,j,u,v) \neq \emptyset$ if and only if $ip^{\alpha} \not\equiv \beta u \pmod{p^e}$ or $jp^{\alpha} \not\equiv \beta v \pmod{p^{n-e}}$ for all $\beta \in \mathbb{Z}$ and $0 \leq \alpha \leq e-1$. In this case, $|\mathcal{N}(i,j,u,v)| = p^n$. Thus,
\[
|\mathfrak{N}_{\mathfrak{S}_q}(G)| = \sum |\mathcal{N}(i,j,u,v)| = \sum_{(i,j,u,v) \in U} |\mathcal{N}(i,j,u,v)| = p^n |B_1|,
\]  
where
\[
B_1 := \{(i, j, u, v) \in U \mid ip^{\alpha} \not\equiv \beta u \pmod{p^e} \text{ or } jp^{\alpha} \not\equiv \beta v \pmod{p^{n-e}} \text{ for all } \beta \in \mathbb{Z}, 0 \leq \alpha \leq e-1 \}.
\]

Now, consider the following set
\begin{align*}
B_2 &= U \setminus B_1 \\
&= \{(i, j, u, v) \in U \mid ip^{\alpha} \equiv \beta u \pmod{p^e} \text{ and } jp^{\alpha} \equiv \beta v \pmod{p^{n-e}} \text{ for some } \beta \in \mathbb{Z}, 0 \leq \alpha \leq e-1 \}.
\end{align*}
Since $U = B_1 \coprod B_2$, we have $|B_1| = |U_1| + |U_2| - |B_2|$. 

Next, we compute the cardinality of $B_2$. Let $(i, j, u, v) \in B_2$. Then, we have
\[
ip^{\alpha} \equiv \beta u \pmod{p^e} \quad \text{and} \quad jp^{\alpha} \equiv \beta v \pmod{p^{n-e}} \text{ for some } \beta \in \mathbb{Z}, 0 \leq \alpha \leq e-1.
\]
Since $\gcd(p, v) = 1$, it follows that $juv^{-1}p^{\alpha} \equiv \beta u \pmod{p^{n-e}}$. This implies $juv^{-1}p^{\alpha} \equiv \beta u \pmod{p^e}$ as $n-e > e$. Thus, $juv^{-1} \equiv i \pmod{p^{e-\alpha}}$, which implies $p \mid p^{e-\alpha} \mid juv^{-1} - i$. Since $p \mid j$ (as $n-2e > 0$), we have $p \mid i$.

Now, we claim that if $p \mid i$, then $(i, j, u, v) \in B_2$. For this, take $\alpha = e-1$ and $\beta = jVp^{e-1}$, where $V = v^{-1} \pmod{p^{n-e}}$. Therefore,
\[
|B_2| = p^{e-1} \varphi(p^e) p^e \varphi(p^{n-e}).
\]

Thus,
\[
|B_1| = p^e \varphi(p^e) p^e \varphi(p^{n-e}) + \varphi(p^e) p^{e-1} p^e \varphi(p^{n-e}) - p^{e-1} \varphi(p^e) p^e \varphi(p^{n-e}) = p^{2e} \varphi(p^e) \varphi(p^{n-e}).
\]

Finally, we find
\[
|\mathfrak{N}_{\mathfrak{S}_q}(G)| = p^n p^{2e} \varphi(p^e) \varphi(p^{n-e}),
\]
and hence,
\[
\mathcal{P}_{f_e} = \dfrac{(p^n)!(p^n-1)!}{p^{2e} \varphi(p^e) \varphi(p^{n-e})},
\]
when $e < \frac{n}{2}$.

\textbf{Case 2:} Now, we assume that $n - e = e$. In this case, if $\mathcal{N}(i, j, u, v) \neq \emptyset$ for some $i, u \in \{0, 1, \ldots, p^e - 1\}$ and $j, v \in \{0, 1, \ldots, p^{n-e} - 1\}$, then $(i, j, u, v) \in U' = \left(V_1' \coprod V_2'\right) \times \left(U_1' \coprod U_2'\right)$, where 
\[
V_1' = \{(i, j) \mid \gcd(j, p^e) = 1\}, \quad \text{and} \quad V_2' = \{(i, j) \mid \gcd(i, p^e) = 1 \text{ and } p \mid j\},
\]
\[
U_1' = \{(u, v) \mid \gcd(v, p^e) = 1\}, \quad \text{and} \quad U_2' = \{(u, v) \mid \gcd(u, p^e) = 1 \text{ and } p \mid v\}.
\]
Thus, we have 
\[
|\mathfrak{N}_{\mathfrak{S}_q}(G)| = \sum |\mathcal{N}(i, j, u, v)| = \sum_{(i, j, u, v) \in U'} |\mathcal{N}(i, j, u, v)|.
\]

Now, let $(i, j, u, v) \in U'$. Since $p^{n-2e} = 1 \mid j$, from Lemma~\ref{Atleast_lemma}, $\mathcal{N}(i, j, u, v) \neq \emptyset$ if and only if 
\[
ip^\alpha \not\equiv \beta u \pmod{p^e} \quad \text{or} \quad jp^\alpha \not\equiv \beta v \pmod{p^e},
\]
for all $\beta \in \mathbb{Z}$ and $0 \leq \alpha \leq e-1$. In this case, $|\mathcal{N}(i, j, u, v)| = p^n$. Note that we can write $U' = B_1' \coprod B_2'$, where 
\begin{equation*}
\begin{split}
B_1' &= \{(i, j, u, v) \in U' \mid ip^\alpha \not\equiv \beta u \pmod{p^e} \text{ or } jp^\alpha \not\equiv \beta v \pmod{p^e} \text{ for all } \beta \in \mathbb{Z}, 0 \leq \alpha \leq e-1\},\\
B_2' &= \{(i, j, u, v) \in U' \mid ip^\alpha \equiv \beta u \pmod{p^e} \text{ and } jp^\alpha \equiv \beta v \pmod{p^e} \text{ for some } \beta \in \mathbb{Z}, 0 \leq \alpha \leq e-1\}.
\end{split}
\end{equation*}

Thus, we have 
\[
|\mathfrak{N}_{\mathfrak{S}_q}(G)| = \sum |\mathcal{N}(i, j, u, v)| = p^n |B_1'|.
\]

Since $U' = B_1' \coprod B_2'$, it follows that $|B_1'| = |U'| - |B_2'|$. Next, we compute the cardinality of the set $B_2'$. To do this, we write $B_2' = C_1' \coprod C_2' \coprod C_3'$, where  
\begin{equation*}
\begin{split}
C_1' &= \{(i, j, u, v) \in B_2' \mid \gcd(v, p) = 1 \text{ and } \gcd(u, p) \neq 1\},\\
C_2' &= \{(i, j, u, v) \in B_2' \mid \gcd(v, p) = 1 \text{ and } \gcd(u, p) = 1\},\\
C_3' &= \{(i, j, u, v) \in B_2' \mid \gcd(v, p) \neq 1 \text{ and } \gcd(u, p) = 1\}.
\end{split}
\end{equation*}

We now compute the cardinalities of $C_1'$, $C_2'$, and $C_3'$ separately in the following subcases:

\textbf{Subcase 2.1:} Firstly, we assume that $(i, j, u, v) \in C_1'$. Then, $\gcd(v, p) = 1$ and $\gcd(u, p) \neq 1$. In this subcase, we have $v^{-1} u j \equiv i \pmod{p^{e-\alpha}}$, which implies $v^{-1} u j \equiv i \pmod{p}$. Since $p \mid u$, it follows that $p \mid i$. Conversely, if $p \mid i$, we can take $\alpha = e-1$ and $\beta = j p^{e-1} y$, where $y = v^{-1} \pmod{p^e}$. It is straightforward to verify that $(i, j, u, v) \in B_2'$, so $(i, j, u, v) \in C_1'$. 

Therefore, if $(i, j, u, v) \in U'$, then $(i, j, u, v) \in C_1'$ if and only if $\gcd(v, p) = 1$, $\gcd(u, p) \neq 1$, and $p \mid i$. Thus, 
\[
|C_1'| = p^{2e-2} \varphi(p^e)^2.
\]

\textbf{Subcase 2.2:} Now, we assume that $(i, j, u, v) \in C_2'$, i.e., $\gcd(v, p) = 1$ and $\gcd(u, p) = 1$. Since $\gcd(v, p) = 1$, we obtain $i \equiv j u v^{-1} \pmod{p^{e-\alpha}}$, which implies $i \equiv j u v^{-1} \pmod{p}$ and $\gcd(i, p) = 1 = \gcd(j, p)$. Conversely, suppose $i \equiv j u v^{-1} \pmod{p}$. Then clearly $(i, j, u, v) \in B_2'$ by choosing $\alpha = e-1$ and $\beta = j p^{e-1} y$, where $y = v^{-1} \pmod{p^e}$. Hence, $(i, j, u, v) \in C_2'$. 

Thus, if $(i, j, u, v) \in U'$, then $(i, j, u, v) \in C_2'$ if and only if $\gcd(v, p) = 1$, $\gcd(u, p) = 1$, and $i \equiv j u v^{-1} \pmod{p}$. For any fixed $j, u, v$, there are $p^{e-1}$ distinct values of $i$ satisfying $i \equiv j u v^{-1} \pmod{p}$. Since $j, u, v$ each have $\varphi(p^e)$ choices, we get 
\[
|C_2'| = p^{e-1} \varphi(p^e)^3.
\]

\textbf{Subcase 2.3:} In this subcase, we assume that $(i, j, u, v) \in C_3'$, so $(i, j, u, v) \in B_2'$, $\gcd(v, p) \neq 1$, and $\gcd(u, p) = 1$. Using a similar argument as in Subcase 2.1, if $(i, j, u, v) \in U'$, then $(i, j, u, v) \in C_3'$ if and only if $\gcd(v, p) \neq 1$, $\gcd(u, p) = 1$, and $p \mid j$. Therefore, 
\[
|C_3'| = p^{2e-2} \varphi(p^e)^2.
\]

From the above subcases, we have 
\[
|B_2'| = |C_1'| + |C_2'| + |C_3'| = p^{2e-2} \varphi(p^e)^2 + p^{e-1} \varphi(p^e)^3 + p^{2e-2} \varphi(p^e)^2.
\]

Thus, 
\[
|B_1'| = |U'| - |B_2'| = p^{2e-1} \varphi(p^e)^2 (p + 1),
\]
and 
\[
|\mathfrak{N}_{\mathfrak{S}_q}(G)| = p^n p^{2e-1} \varphi(p^e)^2 (p + 1).
\]

Finally, 
\[
\mathcal{P}_{f_e} = \frac{(p^n)!(p^n - 1)!}{\varphi(p^e)^2 (p^{2e} + p^{2e-1})},
\]
when $e = \frac{n}{2}$ and $n$ is even.
\end{proof}
\begin{thm}
Let $f_e$ be an $e$-Klenian polynomial and its corresponding permutation polynomial tuple is $\underline{\beta}_{f_e}=(\beta_0,\beta_1,\ldots,\beta_{q-1})\in \mathfrak{S}_q^{q}$. Then there are $q!$ permutation group polynomials equivalent to $f_e$.
\end{thm}
\begin{proof}
For $G$, the group associated to an $e$-Klenian polynomials defined in Lemma \ref{JJ_tpl}, $\mathfrak{C}_{\mathfrak{S}_q}(G)=\mathcal{N}(1,0,0,1)$ and from Lemma \ref{Atleast_lemma}, we have $|\mathcal{N}(1,0,0,1)|=p^n=q$. Thus, $|\mathfrak{C}_{\mathfrak{S}_q} (G)|=q$ and the result follows from Theorem \ref{Equivalent_PGP}. 
  \end{proof} 

\section{Concluding Remarks and Future Directions}\label{S5}

Bivariate local permutation polynomials over the finite field $\F_q$ have attracted significant interest, since they are in one-to-one correspondence with Latin squares of order $q$. In this work, we constructed a family of permutation group polynomials, a special class of bivariate local permutation polynomials, over the finite field $\F_q$ of arbitrary characteristic, and we also provided explicit expression for their companion.

In \cite{JJ_2023}, the authors posed the problem of counting the $e$-Klenian polynomials for all $e \ge 1$, noting that this is a nontrivial problem. We completely resolve this problem. Motivated by this, we further determined the exact number of permutation group polynomials that are of the form constructed in this paper. We also computed the number of permutation group polynomials equivalent to our proposed family, and those equivalent to 
$e$-Klenian polynomials.

As noted earlier, the construction of permutation group polynomials is purely a group-theoretic problem. In particular, it remains open to develop a systematic approach for identifying all subgroups $G < \mathfrak{S}_q$ such that
\begin{itemize}
    \item[(a)] $|G| = q$, and
    \item[(b)] $g(c) \ne c$ for all $c \in \F_q$ and for all $g \in G$, $g \ne I$.
\end{itemize}
To date, only a few such subgroups of $\mathfrak{S}_q$ are known, highlighting the need for further research in this direction. In particular, for small $n$, one may attempt to classify such subgroups of order $p^n$ using the existing knowledge of isomorphism classes of groups of order $p^n$.

Another direction is to investigate the problem of enumerating all companions of a given bivariate local permutation polynomial over $\F_q$.

\section*{Acknowledgments}
We sincerely thank the editors for handling our paper, and the referees for their careful reading, valuable comments, and constructive suggestions.

\end{document}